\documentclass[reqno, 12pt]{amsart}
\pdfoutput=1
\makeatletter
\let\origsection=\section \def\section{\@ifstar{\origsection*}{\mysection}} 
\def\mysection{\@startsection{section}{1}\z@{.7\linespacing\@plus\linespacing}{.5\linespacing}{\normalfont\scshape\centering\S}}
\makeatother        
\usepackage{amsmath,amssymb,amsthm}    
\usepackage{mathabx}\usepackage{wasysym}
\usepackage{mathrsfs}
\usepackage{bbm, accents} 

\usepackage{xcolor}  	
\usepackage[backref]{hyperref}
\hypersetup{
	colorlinks,
    linkcolor={red!60!black},
    citecolor={green!60!black},
    urlcolor={blue!60!black},
}

\usepackage{bookmark}

\usepackage[abbrev,msc-links,backrefs]{amsrefs} 
\usepackage{doi}

\renewcommand{\PrintDOI}[1]{\doi{#1}}

\usepackage[T1]{fontenc}
\usepackage{lmodern}

\usepackage[english]{babel}
\usepackage{microtype}

\usepackage{tikz}
\usepackage{graphicx}              \usepackage{caption}
\usepackage{subcaption}

\theoremstyle{plain}
\newtheorem{thm}{Theorem}[section]

\newtheorem{prop}[thm]{Proposition}

\newtheorem{cor}[thm]{Corollary}

\newtheorem{lem}[thm]{Lemma}

\theoremstyle{definition}
\newtheorem{rem}[thm]{Remark}
\newtheorem{dfn}[thm]{Definition}

\usepackage{geometry}
\numberwithin{equation}{section}

\let\theta=\vartheta
\let\rho=\varrho
\let\phi=\varphi
\let\til=\widetilde
\let\wh=\widehat

\def\tim{{\hskip-.17em\times}}

\def\cL{{\mathcal L}}
\def\cQ{{\mathcal Q}}
\def\cR{{\mathcal R}}
\def\cS{{\mathcal S}}

\def\ccL{{\mathscr{L}}}
\def\ccC{\mathscr{C}}
\def\ccS{\mathscr{S}}
\def\ccT{\mathscr{T}}
\def\ccP{\mathscr{P}}

\def\gZ{\mathfrak{Z}}
\def\gY{\mathfrak{Y}}
\def\gW{\mathfrak{W}}

\DeclareMathOperator{\HJ}{HJ}
\DeclareMathOperator{\HJC}{HJC}

\usepackage{enumitem} 
\def\rmlabel{\upshape({\itshape \roman*\,})}

\makeatletter
\def\greek#1{\expandafter\@greek\csname c@#1\endcsname}
\def\Greek#1{\expandafter\@Greek\csname c@#1\endcsname}
\def\@greek#1{\ifcase#1
	\or $\alpha$%
	\or $\beta$%
	\or $\gamma$%
	\or $\delta$%
	\or $\epsilon$%
	\or $\zeta$%
	\or $\eta$%
	\or $\theta$%
	\or $\iota$%
	\or $\kappa$%
	\or $\lambda$%
	\or $\mu$%
	\or $\nu$%
	\or $\xi$%
	\or $o$%
	\or $\pi$%
	\or $\rho$%
	\or $\sigma$%
	\or $\tau$%
	\or $\upsilon$%
	\or $\phi$%
	\or $\chi$%
	\or $\psi$%
	\or $\omega$%
\fi}
\def\@Greek#1{\ifcase#1
	\or $\mathrm{A}$%
	\or $\mathrm{B}$%
	\or $\Gamma$%
	\or $\Delta$%
	\or $\mathrm{E}$%
	\or $\mathrm{Z}$%
	\or $\mathrm{H}$%
	\or $\Theta$%
	\or $\mathrm{I}$%
	\or $\mathrm{K}$%
	\or $\Lambda$%
	\or $\mathrm{M}$%
	\or $\mathrm{N}$%
	\or $\Xi$%
	\or $\mathrm{O}$%
	\or $\Pi$%
	\or $\mathrm{P}$%
	\or $\Sigma$%
	\or $\mathrm{T}$%
	\or $\mathrm{Y}$%
	\or $\Phi$%
	\or $\mathrm{X}$%
	\or $\Psi$%
	\or $\Omega$%
\fi}

\AddEnumerateCounter{\greek}{\@greek}{24}
\AddEnumerateCounter{\Greek}{\@Greek}{12}

\makeatother

\def\glabel{\upshape({\itshape \greek*})}

\let\phi=\varphi
\def\cP{{\mathcal P}}

\def\lra{\longrightarrow}

\def\bl{\bigr(}
\def\br{\bigr)}

\let\vn=\varnothing

\newcommand{\seq}[1]{\accentset{\rightharpoonup}{#1}}

\newtheoremstyle{note}  {4pt}  {4pt}  {\sl}  {}  {\itshape}  {.}  {.5em}          {}
\theoremstyle{note}
\newtheorem{claim}{Claim}

\makeatletter
\def\moverlay{\mathpalette\mov@rlay}
\def\mov@rlay#1#2{\leavevmode\vtop{   \baselineskip\z@skip \lineskiplimit-\maxdimen
   \ialign{\hfil$\m@th#1##$\hfil\cr#2\crcr}}}
\newcommand{\charfusion}[3][\mathord]{
    #1{\ifx#1\mathop\vphantom{#2}\fi
        \mathpalette\mov@rlay{#2\cr#3}
      }
    \ifx#1\mathop\expandafter\displaylimits\fi}
\makeatother

\newcommand{\dcup}{\charfusion[\mathbin]{\cup}{\cdot}}

\let\vn=\varnothing
\def\lra{\longrightarrow}

\def\bl{\bigr(}
\def\br{\bigr)}

\newcommand{\str}{\scalebox{1.1}[1.1]{{$\blacktriangleleft$}}}
\newcommand{\Str}{\hskip.4em {\scalebox{1.4}[1.4]{{$\blacktriangleleft$}}} \hskip.4em}
\newcommand{\mydcup}{\scalebox{1.1}[1.1]{$\,\dcup\,$}}

\let\sm=\setminus

\def\conc{\mathop{\hexstar}}

\begin{document}
\title[A Ramsey Class for Steiner Systems]{A Ramsey Class for Steiner Systems}

\author[Vindya Bhat]{Vindya Bhat}
\address{Department of Mathematics, 
New York University, New York, NY 10012-1110, USA}
\email{vbhat@cims.nyu.edu}

\author[Jaroslav Ne\v{s}et\v{r}il]{Jaroslav Ne\v{s}et\v{r}il}
\address{Department of applied Mathematics and 
Institute of Theoretical Computer Science,  
Charles University, 11800 Praha 1, Czech Republic}
\email{nesetril@kam.ms.mff.cuni.cz}
\thanks{The second author was supported by grant 
CE-ITI P202/12/G061 of GA\v{C}R}
\thanks{The second and fourth author 
were supported by grant ERC-CZ LL1201 of the Czech Ministry of Education.}

\author[Christian Reiher]{Christian Reiher}
\address{Fachbereich Mathematik, Universit\"at Hamburg,
Bundesstra\ss{}e~55, D-20146 Hamburg, Germany}
\email{Christian.Reiher@uni-hamburg.de}

\author[Vojt\v{e}ch R\"{o}dl]{Vojt\v{e}ch R\"{o}dl}
\address{Department of Mathematics and Computer Science, 
Emory University, Atlanta, GA 30322, USA}
\email{rodl@mathcs.emory.edu}
\thanks{The fourth author was supported by NSF grant DMS 1301698.}


\begin{abstract}
We construct a Ramsey class whose objects are Steiner systems. 
In contrast to the situation with general $r$-uniform hypergraphs, it turns out
that simply putting linear orders on their sets of vertices is not enough for this 
purpose: one also has to strengthen the notion of subobjects used from 
``induced subsystems'' to something we call ``strongly induced subsystems''. 

Moreover we study the Ramsey properties of other classes of Steiner systems obtained 
from this class by either forgetting the order or by working with the usual notion 
of subsystems. This leads to a perhaps surprising induced Ramsey theorem in which 
{\it designs} get coloured.   
\end{abstract} 

\maketitle

\section{Introduction} \label{sec:intro}

\subsection{Structural Ramsey Theory} \label{subsec:SRT}

The Entscheidungsproblem from mathematical logic led
Ramsey~\cite{Ramsey30} to the following combinatorial principle that
became a cornerstone of an area now called Ramsey theory.

\begin{thm}\label{thm:Ramsey}
For any positive integers $m$, $r$, and $c$ there exists a positive integer $M$ with
\[
	M\lra (m)^r_c\,,
\]
i.e., such that no matter how the $r$-element subsets of an $M$-element set $X$ 
are coloured with~$c$ colours, there will always be an $m$-element subset $Y$ 
of $X$ such that all $r$-element subsets of~$Y$ are the same colour.
\end{thm}  

This result admits a standard reformulation in the language of hypergraphs: for a fixed 
integer $r\ge 2$, an {\it $r$-uniform hypergraph}, or {\it $r$-graph} for short, 
is a pair $G=(V,E)$ consisting of a vertex set $V$ and an edge set $E \subseteq \binom{V}{r}$. 
If $G=(V, E)$ is an $r$-graph, we use the standard notation of writing $V(G)=V$, 
$v_G=\vert V \vert$, $E(G)=E$, and $e_G=\vert E\vert$. 
In case that~$G$ has all possible edges, that is if $E=\binom{V}{r}$, we say that $G$ is 
{\it a clique} and if additionally $\vert V\vert=m$ we say that $G$ is a $K_m^{(r)}$. 

Now~Theorem~\ref{thm:Ramsey} informs us  that if $M$ is large enough 
depending on $m$, $r$, and $c$, then
\[
	K_M^{(r)}\lra \bigl(K_m^{(r)}\bigr)^e_c
\]
holds, meaning that in every edge-colouring of $K_M^{(r)}$ using $c$ colours there
occurs a monochromatic copy of $K_m^{(r)}$. 
This raises the question whether there is a similar result when the target hypergraph 
$K_m^{(r)}$ is replaced by an arbitrary $r$-graph $G$. Of course, when we just ask 
for a monochromatic appearance of $G$ as a subhypergraph we may apply Ramsey's theorem 
with $m=v_G$. 
But the problem becomes significantly more challenging when we ask for a monochromatic 
induced copy of $G$.

Here, for two given $r$-graphs $G$ and $H$, we say that $G$ is an {\it induced subhypergraph} 
of $H$ and write $G\le H$ if $V(G) \subseteq V(H)$ and $E(G) = E(H) \cap \binom{V(G)}{r}$. 
By an {\it induced copy} of $G$ in~$H$ we mean an induced subhypergraph $\widetilde{G}$
of $H$ that is isomorphic to $G$. 

\begin{thm}[Induced Ramsey Theorem]\label{thm:rind}
Given any $r$-uniform hypergraph $G$ and any number $c\ge 1$ of colours, there exists an 
$r$-uniform hypergraph $H$ with
\begin{equation}\label{eq:e-arrow}
H\lra (G)^e_c\,
\end{equation}
in the sense that for every colouring of the edges of $H$ with $c$ colours there
exists a monochromatic induced copy of $G$ in $H$.
\end{thm}

For $r=2$ this was proved independently by
Deuber~\cite{Deuber75}, by Erd\H{o}s, Hajnal, and P\'{o}sa~\cite{EHP75}, 
and by R\"{o}dl in his master thesis~\cites{Rodl73, Rodl76}. For an alternative
proof we refer to~\cite{NeRo3a}. That article introduces the
so-called {\it partite method} on which much of the subsequent progress in 
this area is based, including the present work. There are even ``arithmetic'' 
applications of this method, see e.g., the work of Leader and 
Russell~\cites{LR06}.

For general hypergraphs the proof of Theorem~\ref{thm:rind} was obtained 
independently by Abramson and Harrington~\cite{AH78} and by Ne\v{s}et\v{r}il
and R\"{o}dl~\cite{NeRo1}. Actually both of these articles prove considerably 
stronger results. Shorter proofs of Theorem~\ref{thm:rind} utilising 
partite structures may be found in~\cite{NeRo5} and~\cite{BR16}.

The next level of generality is obtained by replacing the edge-symbol $e$ in 
formula~\eqref{eq:e-arrow} by other $r$-graphs. Such considerations may 
actually take place in a more abstract context that we are going to introduce
next. 

\begin{dfn} \label{dfn:rcl}
Let $\mathscr{C}$ be a class of objects endowed with an equivalence relation 
called  {\it isomorphism} and with a transitive {\it subobject} relation. Given two 
objects $F$ and $G$ from $\mathscr{C}$ we write $\binom{G}{F}$ for the 
class of all subobjects of $G$ that are isomorphic to $F$. 

For three objects $F, G, H\in\ccC$ and a positive integer $c$ the partition symbol
\[
	H\lra(G)^F_c
\]
means that no matter how $\binom{H}{F}$ gets coloured with $c$ colours, there
is some $\widetilde{G} \in \binom{H}{G}$ for which~$\binom{\widetilde{G}}{F}$ 
is monochromatic.  

The class $\ccC$ is said to have the {\it $F$-Ramsey property} if for every 
$G \in \mathscr{C}$ and every $c$ there exists an $H\in\ccC$ with $H\lra(G)^F_c$.

Finally, $\ccC$ is a {\it Ramsey class} if it has the $F$-Ramsey property 
for every $F\in\ccC$.
\end{dfn}

For example, Ramsey's theorem asserts that the class of finite sets with isomorphisms 
being bijections and subobjects being subsets is a Ramsey class. 
Moreover, the induced Ramsey theorem tells us that the class $\ccC^{(r)}$ 
of all $r$-uniform hypergraphs with the usual notion of isomorphism and whose
subobjects are induced subhypergraphs has the $e$-Ramsey property, where $e=K_r^{(r)}$.
It was proved in~\cite{NeRo1} that $\ccC^{(r)}$ has the $F$-Ramsey
property if~$F$ is either a clique or discrete, i.e., edgeless (see also~\cite{AH78}). 
That this condition
on $F$ is also necessary was proved for $r=2$ in~\cite{NeRo75} and it seems to belong
to the folklore of the subject that the probabilistic approach of~\cite{NeRo2} 
yields the same result for all $r\ge 2$ (see also Theorem~\ref{thm:main-b} below). 
 
So $\ccC^{(r)}$ is not a Ramsey class, but it turns out that a slight variant of
this class is Ramsey: let $\ccC^{(r)}_{<}$ be the class of all 
{\it ordered $r$-uniform hypergraphs}, that is, $r$-graphs endowed with a fixed linear 
ordering of their vertices. 
The isomorphisms of $\ccC^{(r)}_{<}$ are required to respect these orderings. 
Thus for $F_<, G_<\in \ccC^{(r)}_{<}$ with underlying (unordered) $r$-graphs
$F$ and $G$ the set~$\binom{G_<}{F_<}$ may in general correspond to a proper subset of 
$\binom{G}{F}$.

\begin{thm} \label{thm:hrc}
The class $\ccC^{(r)}_{<}$ of all ordered $r$-uniform hypergraphs is a Ramsey class. 
\end{thm} 

Again this was proved in~\cite{AH78} and~\cite{NeRo1} (see also~\cite{Nero7}). 
Other known examples of Ramsey classes include 
finite vector spaces over a fixed field $F$ (see~\cite{GLR72}), 
and finite partially ordered sets with fixed linear extensions~\cites{NeRo6, NeRo8, PTW85}.  
The study of Ramsey classes found its revival after several decades due to 
its connection with topological dynamics via ultrahomogeneous  structures
described by Kechris, Pestov, and Todorcevic in their seminal paper~(see~\cite{KPT}).
Further results in this direction were obtained in the recent work of 
van Th\'{e}~\cites{The10, The13} and by Hubi\v{c}ka and Ne\v{s}et\c{r}il~\cites{HN1, HN2}. 
A very readable account of the 
Kechris-Pestov-Todorcevic correspondence has recently been given by Solecki in his 
survey~\cite{Solecki}. For more information on structural Ramsey theory in general
we refer to Bodirsky's survey chapter~\cite{Bod15}.

\subsection{Steiner Systems} 
\label{subsec:SS}

Throughout the rest of this article, we consider classes of $r$-graphs called 
Steiner systems. For fixed integers $r \ge t\ge 2$, by a {\it Steiner $(r,t)$-system} 
we mean an $r$-uniform hypergraph $G=(V, E)$ with the property that every $t$-element 
subset of $V$ is contained in at most one edge of $G$. 

Such objects are also called ``partial Steiner systems'' in the design-theoretic literature, 
while the term ``Steiner system'' is reserved there to what we will call 
``complete Steiner systems'' (see Definition~\ref{dfn:cmpl} below). 

We denote the class of all Steiner $(r,t)$-systems with subobjects again being induced 
subhypergraphs by $\mathscr{S}(r,t)$. 
For example, the members of
$\mathscr{S}(r,2)$ are sometimes referred to as {\it linear hypergraphs} while
$\mathscr{S}(r,r)=\ccC^{(r)}$. 

The following generalisation of the induced Ramsey theorem was obtained by 
Ne\v{s}et\v{r}il and R\"{o}dl in~\cite{NeRo4}.

\begin{thm} \label{thm:edgesteiner} 
For any integers $r\ge t\ge 2$ the class $\mathscr{S}(r,t)$ has the edge-Ramsey property. 
\end{thm}

As remarked in~\cite{NeRo4} the proof described there does also show that the 
corresponding ordered class $\mathscr{S}_<(r,t)$ has the edge-Ramsey property.
The members of this class are, of course, Steiner $(r,t)$-systems with fixed 
linear orderings of their vertex sets and the isomorphisms of $\mathscr{S}_<(r,t)$ 
are required to preserve these orderings.

In the light of Theorem~\ref{thm:hrc} it is natural to wonder whether these classes
$\mathscr{S}_<(r,t)$ are Ramsey classes as well, but it turns out that for $t<r$ 
they are not (see also Corollary~\ref{cor:ind-st-ord} below). 

The main result of this article, however, asserts that those classes can be made Ramsey by changing the subobject relation as follows.

\begin{dfn} \label{dfn:strong}
Given two Steiner $(r,t)$-systems $G$ and $H$, we say that $G$ is a 
{\it strongly induced subsystem} of $H$ and write $G\Str H$ if 
\begin{enumerate}[label=\rmlabel]
\item\label{it:str-a} $G\le H$, i.e., $G$ is an induced subsystem of $H$ 
\item\label{it:str-b} and moreover $\vert e \cap V(G)\vert < t$ holds for all 
$e\in E(H)\sm E(G)$.
\end{enumerate}
The set of all strongly induced copies of $G$ in $H$ is denoted by $\binom{H}{G}_{\str}$.  
\end{dfn}

This concept leads to two further classes of Steiner systems.

\begin{dfn} \label{dfn:str-cl}
Let $\mathscr{S}^{\,\str}(r,t)$ be the class of all Steiner $(r,t)$-systems with isomorphisms 
as usual and whose subobjects are strongly induced subsystems. 
Similarly, $\mathscr{S}^{\,\str}_<(r,t)$ refers to the corresponding class of ordered 
Steiner $(r,t)$-systems.  
\end{dfn}
 
We may now announce the first main result of this article.

\begin{thm} \label{thm:main-a} 
The class $\mathscr{S}^{\,\str}_<(r,t)$ is Ramsey.  
\end{thm} 
 
As said above, the interest in Ramsey classes did recently increase due to 
their connection with other areas.
But it may be observed that even when one only cares about colouring edges,
Theorem~\ref{thm:main-a} gives more information than the earlier Theorem~\ref{thm:edgesteiner}.
Specifically, we obtain a monochromatic {\it strongly} induced copy of the original
Steiner system, and thus a copy that interacts only very little with any other such 
copy. Besides being of interest, such a strong type of containment of the monochromatic copy 
seems to be a useful property for further proofs utilising the partite method.  
 
In order to gain a better understanding as to why 
Theorem~\ref{thm:main-a} is true, whilst in general the 
class~$\mathscr{S}(r,t)$ fails to be Ramsey, we will also determine 
for each of the four classes of Steiner systems introduced so far for which objects 
$F$ they have the $F$-Ramsey property. 
In this manner we can separately study the effects of ordering the vertices
and of changing the definition of subobjects. 

The adjectives {\it weak}, {\it strong}, {\it unordered}, and {\it ordered}
will be applied to these classes as indicated by the following table.
 
\medskip
 
\begin{center}
\begin{tabular}{|c||c|c|}
\hline
      & unordered & ordered \\ \hline \hline
weak & $\mathscr{S}(r,t)$ & $\mathscr{S}_<(r,t)$ \\ \hline
strong & $\mathscr{S}^{\,\str}(r,t)$ & $\mathscr{S}^{\,\str}_<(r,t)$  \\ \hline
\end{tabular}
\end{center}
 
\medskip

It will turn out that the property of Steiner systems relevant to unordered classes is that 
of homogeneity. 

\begin{dfn}\label{dfn:rigid}
A Steiner $(r, t)$-system $F$ is said to be 
{\it homogeneous}\footnote[1]{The reader familiar with model theory might find this 
usage of the word ``homogeneous'' confusing, because it means something entirely else 
in this context. In the theory of finite graphs, however, it is often used  
with the same meaning as above.}
if every 
permutation of~$V(F)$ induces an automorphism of~$F$. 
\end{dfn}
Thus a member of $\ccS(r,r)=\ccC^{(r)}$ is homogeneous if 
either it is of the form $K^{(r)}_m$ for some $m\ge r$ or if it has no edges. 
It should be observed, however, that except for the edge~$e$ the homogeneous $r$-graphs of the 
former type do not belong to $\ccS(r, t)$ as long as $t<r$.

Similarly for $t<r$ the weak classes demand completeness.

\begin{dfn}\label{dfn:cmpl}
A Steiner $(r, t)$-system $F$ is called
{\it complete} if for every $x\in\binom{V(F)}{t}$ there is an edge $e\in E(F)$ with 
$x\subseteq e$. 
\end{dfn}
Trivial cases for this to happen are that $F$ consists of less than $t$
isolated vertices or that~$F=e$. Non-trivial examples of complete Steiner systems, 
also known as {\it designs} in the literature, are very hard to come up with 
(at least for $t>7$, say). But recently Keevash~\cite{Keevash} established 
an important 160 year old conjecture of Steiner regarding their existence.       
 
With this terminology the second main result of this article reads as follows.

\begin{thm} \label{thm:main-b} 
Let $r\ge t\ge 2$, $\ccT\in\bigl\{\ccS, \ccS_<, \ccS^{\,\str}, \ccS^{\,\str}_<\bigr\}$, 
and $F\in\ccT(r,t)$.
Then $\ccT(r,t)$ has the $F$-Ramsey property if both of the following conditions hold:
\begin{enumerate}[label=\rmlabel]
\item\label{it:main-uno}  If $\ccT(r,t)$ is unordered, then $F$ is homogeneous, 
\item\label{it:main-weak} If $\ccT(r,t)$ is weak and $t < r$, then $F$ is a complete
Steiner system. 
\end{enumerate}
If on the other hand \ref{it:main-uno} or \ref{it:main-weak} fails, 
then $\ccT(r,t)$ fails to have the $F$-Ramsey property.
\end{thm}

Notice that this generalises all other structural Ramsey-theoretic theorems 
stated above. Theorem~\ref{thm:main-a} 
follows, as the above clauses~\ref{it:main-uno} and~\ref{it:main-weak} hold vacuously 
for $\ccT=\ccS^{\,\str}_<$.  The case $\ccT=\ccS$ and $F=e$ yields 
Theorem~\ref{thm:edgesteiner}, and the induced Ramsey theorem 
corresponds to the case $r=t$, $\ccT=\ccS$, and $F=e$. 

For the readers' convenience we would now like to state more explicitly what 
Theorem~\ref{thm:main-b} says for $\ccT\in\bigl\{\ccS, \ccS_<, \ccS^{\,\str}\bigr\}$.

\begin{cor} \label{cor:ind-st} 
The class $\mathscr{S}(r,t)$ has the $F$-Ramsey property if and only if 
one of the following two cases holds:
\begin{enumerate}
\item[$\bullet$] $r>t$ and either $F$ is an edge or $v_F < t$, 
\item[$\bullet$] $r=t$ and $F$ is homogeneous. 
\end{enumerate}
\end{cor}

\begin{cor}  \label{cor:ind-st-ord} 
The class $\ccS_<(r,t)$ has the $F$-Ramsey property if and only if $r=t$ or 
$F$ is a complete Steiner $(r, t)$-system (the cases $v_F<t$ and $F=e$ are included). 
\end{cor}

\begin{cor} \label{cor:ind-st-str} 
The class $\mathscr{S}^{\,\str}(r,t)$ has the $F$-Ramsey property if and only if 
\begin{enumerate}
\item[$\bullet$] $F$ is an edge 
\item[$\bullet$] or $F$ is discrete, 
\item[$\bullet$] or if $r=t$ and $F$ is a clique.  
\end{enumerate}
\end{cor}

\begin{rem}
At first glance it might be tempting to think that Theorem~\ref{thm:main-a} could follow 
from the main result of~\cite{NeRo1}, because (1) Steinerness is 
describable in terms of forbidden substructures and (2) after adding
a new $t$-ary predicate to be interpreted by the $t$-sets of 
vertices which are contained in an edge, 
{\it strongly} induced subobjects become ordinary subobjects.

However, in the case at hand the forbidden substructures are not 
irreducible and thus Theorem~\ref{thm:main-a} is not a consequence of~\cite{NeRo1}. 
(On the other hand, the two results do not contradict each other because~\cite{NeRo1} 
speaks about {\it ideal} subcategories only, and this assumption 
does not apply here. Roughly speaking, this is because the product of a Steiner system 
with an arbitrary hypergraph does not need to be Steiner.) 
\end{rem}

\section{Why is the strong ordered class Ramsey?} \label{sec:pos}

\subsection{Overview} 
As our proof of Theorem~\ref{thm:main-a} is not so short and proceeds in several steps,
we would like to begin with an informal discussion of some of its central ideas and how 
they relate to two earlier arguments yielding special cases, namely to the first of the two
proofs of the induced Ramsey theorem presented in~\cite{NeRo5} and to the proof of 
Theorem~\ref{thm:edgesteiner} from~\cite{NeRo4}.  

A common theme occurring in all these proofs is that they rely on the 
{\it partite construction} introduced in~\cite{NeRo3a}. Generally speaking,
this construction is a versatile iterative amalgamation technique that allows
in many situations to strengthen ``weak'' Ramsey theoretic facts. 
For instance, Ramsey's theorem trivially yields a version of 
Theorem~\ref{thm:rind} where we just demand the monochromatic copy of $G$ to be
a subhypergraph that is not necessarily induced. Still, this can be used in a partite 
construction generating a bigger Ramsey object for $G$ of the kind that is actually 
desired by Theorem~\ref{thm:rind}. 

The two main ingredients of such a partite construction are 
\begin{enumerate}
\item[$\bullet$] a so-called {\it partite lemma}
\item[$\bullet$] and {\it partite amalgamations}.
\end{enumerate}

In both~\cite{NeRo5} and~\cite{NeRo4} the partite lemma essentially says 
that the result to be proved holds whenever $G$ is $r$-partite and it is proved
by an application of the  Hales-Jewett Theorem discussed in 
Subsection~\ref{subsec:HJ} below.

The partite amalgamations were straightforward in~\cite{NeRo5} but in the 
case of Steiner systems~\cite{NeRo4} some arguments were needed to show that 
no pairs of distinct edges intersecting in $t$ or more vertices were created. 
This verification was in turn based on a study of the intersection properties
of so-called ``canonical copies'' that correspond to the combinatorial lines of
the Hales-Jewett cube (see~\cite{NeRo4}*{Lemma~2.8}).

When one attempts to base a proof of Theorem~\ref{thm:main-a} on the same ideas,
there arises the following problem: While it is perfectly possible to use
the Hales-Jewett theorem again for proving a partite Lemma (see the 
``preliminary partite lemma'' in Subsection~\ref{subsec:PPL} below)
it turns out that its intersection properties are not strong enough for enabling us  
to perform partite amalgamations in sufficiently general situations. 
Our way of coping with that difficulty is that we run the partite construction 
twice. The first time we just aim at getting a better partite lemma, called 
the ``clean partite lemma'' in Subsection~\ref{subsec:PL} below. 
Owing to its more useful intersection properties this lemma 
can then be applied in a second partite construction that proves
Theorem~\ref{thm:main-a} (see Subsection~\ref{subsec:APC} below).

For experts on the partite method it might also be interesting to
observe that while our proof of the clean partite lemma uses a special
feature of the Hales-Jewett cube (see clause~\ref{it:ppl2} of Lemma~\ref{subsec:PPL} below), 
it only does so ``in the projection''. 

Throughout the rest of this section we fix integers $r\ge t\ge 2$ and a number of 
colours~$c\ge 1$.

\subsection{The Hales-Jewett Theorem} \label{subsec:HJ}

In this short subsection we will fix some terminology regarding the Hales-Jewett
theorem. For a finite nonempty set $\cQ$ and a positive integer~$n$, the Hales-Jewett
cube $\HJC(\cQ, n)$ is defined to be the Cartesian power $\cQ^n$. For any $h\in [n]$
we let $\pi_h\colon \HJC(\cQ, n)\lra \cQ$ denote the natural projection onto the $h$-th 
coordinate defined by $\pi_h(Q_1, \ldots, Q_n)=Q_h$ for every 
$(Q_1, \ldots, Q_n)\in\HJC(\cQ, n)$. 

Now consider all partitions $[n]=C\dcup M$ with $M\not =\vn$ and for each of them all 
functions $g\colon C\lra\cQ$. We will denote the set of all pairs $(C, g)$ arising in 
this way by $\ccL(\cQ, n)$. 
Its members encode in the following way so-called combinatorial lines that are subsets
of~$\HJC(\cQ, n)$.
For each $(C, g)\in\ccL(\cQ, n)$ we define the embedding 
\[
	\eta_{C, g}\colon \cQ\lra\HJC(\cQ, n)
\]
such that for every $h\in [n]$ and $Q\in \cQ$ we have
\begin{equation} \label{eq:HJ-eta}
	\pi_h\bl\eta_{C, g}(Q)\br
		=\begin{cases}
			g(h) & \text{ if } h\in C \cr
			Q    & \text{ if } h\in M.
		\end{cases}
\end{equation}

The range of this map, i.e., the set $\eta_{C, g}[\cQ]$,
is called the {\it combinatorial line} associated with the pair $(C, g)$. It is convenient
to think of the coordinates from $C$ as being {\it constant} and of those from $M$ as being
{\it moving}.

We may now state the Hales-Jewett theorem~\cite{HJ63} (see also~\cite{Sh329} for a 
beautiful alternative proof).
 
\begin{thm}[Hales-Jewett] 
For any positive integers $q$ and $c$ there exists a positive integer $\HJ(q, c)$
such that whenever $\cQ$ is a $q$-set and $n\ge \HJ(q, c)$ is an integer, every 
colouring of $\HJC(\cQ, n)$ with $c$ colours contains a monochromatic combinatorial line.
\end{thm}

\subsection{The preliminary partite lemma} \label{subsec:PPL}
Let $k$ be any positive integer.
A Steiner $(r, t)$-system~$X$ is said to be {\it $k$-partite} if its vertex set may 
be partitioned into $k$ classes $V^1, \ldots , V^k$ such that all edges $e \in E(X)$ 
are {\it crossing} in the sense that $|e \cap V^i| \le 1$ holds for all $i\in [k]$.  
We shall write $\ccP(k, r, t)$ for the class of all $k$-partite Steiner $(r,t)$-systems 
$X=\bl (V^i)_{i=1}^k,E\br$ having one such $k$-partition of its vertex set distinguished. 
Of course, if $k<r$ then such a $k$-partite Steiner system cannot have any edges. But the 
additional structure on such partite Steiner systems we need to deal with below may still
be non-trivial.

Associated with each $X\in \ccP(k, r, t)$ we have a {\it projection}  
\begin{equation} \label{eq:psi}
	\psi_X\colon V(X)\lra [k]
\end{equation}
sending the vertices from $V^i(X)$ to $i$ for every $i\in [k]$. 
In terms of this map, the crossing property of the edges means that $\psi_X$ is 
injective on every edge of~$X$. 

Let us fix some $F\in \ccS(r,t)$ with $V(F)=[k]$ for the rest of this subsection.
Given any $X\in \ccP(k, r, t)$ we shall write $\binom{X}{F}^\tim_{\str}$
for the set of all those strongly induced copies $\til{F}$ of~$F$ in~$X$
which have the property that $\psi_X$ is an isomorphism from $\til{F}$ to $F$.
In particular, these copies of $F$ are ``crossing'' in $X$ and contain precisely one 
vertex from each set of the form~$V^i(X)$. The little cross ``$\times$'' in the notation
$\binom{X}{F}^\tim_{\str}$ is intended to remind us of this fact. 
  
We are now ready to define the class of objects to which our partite lemmata will apply.

\begin{dfn} \label{dfn:Fhyp}
For $F\in\ccS(r, t)$ an {\it $F$-hypergraph} is a pair $(X, \cQ)$ with $X\in \ccP(k, r, t)$ 
and~$\cQ\subseteq \binom{X}{F}^\tim_{\str}$, such that 
\begin{equation} \label{eq:225}
	\psi_X[e]\in E(F) \text{ holds for every } e\in E(X)\,.
\end{equation}
The class of all $F$-hypergraphs is denoted by $\ccP(F, r, t)$.
\end{dfn} 

\begin{figure}[h]
\includegraphics[scale=1]{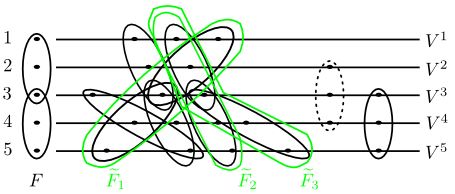}
\caption{Example of an $F$-hypergraph. The dotted ``edge'' must not exist.
Notice that not all copies of $F$ are in $\cQ$ (green).}
\end{figure}

About two $F$-hypergraphs $(X, \cQ)$ and $(Y, \cR)$ we say that $(X, \cQ)$ is 
{\it strongly induced} in~$(Y, \cR)$ and write $(X, \cQ)\Str (Y, \cR)$ if $X\Str Y$ and 
$\cQ=\binom{X}{F}^\tim_{\str}\cap\cR$ hold. By $\binom{(Y, \cR)}{(X, \cQ)}$ we denote the set 
of all strongly induced copies $(\til{X}, \til{\cQ})$ of $(X, \cQ)$ in $(Y, \cR)$. 
If $\gY\subseteq \binom{(Y, \cR)}{(X, \cQ)}$ is a system of such copies and $c$ is a 
positive integer, then the partition symbol 
\begin{equation}\label{eq:gY}
	\gY\lra (X, \cQ)^F_c
\end{equation}
means that for every colouring of $\cR$ with $c$ colours there exists a copy  
$(\til{X}, \til{\cQ})\in\gY$ such that~$\til{\cQ}$ is monochromatic. Evidently 
if $(Y, \cR)$ supports any such system
$\gY\subseteq \binom{(Y, \cR)}{(X, \cQ)}$, then 
\[
	\binom{(Y, \cR)}{(X, \cQ)}\lra (X, \cQ)^F_c
\]
holds as well. However, we may still gain something by considering the more general
partition property~\eqref{eq:gY} because a system of copies $\gY$ for which it is valid
might also have some additional properties which the full system 
$\binom{(Y, \cR)}{(X, \cQ)}$ possibly lacks. These additional properties can then 
be exploited in future arguments. For example, clause~\ref{it:ppl2} of the preliminary 
partite lemma stated below is not hard to verify for the system of copies~$\gY$ we get from 
the Hales-Jewett theorem, but it seems to be less clear whether $\binom{(Y, \cR)}{(X, \cQ)}$ 
would satisfy it as well. Roughly speaking, this clause asserts that the property of 
a $t$-set to be contained in a relevant copy of $F$ reflects from $(Y, \cR)$ to members 
of $\gY$. 
 
\begin{lem}[Preliminary partite lemma] \label{lem:ppl}
For every $F$-hypergraph $(X, \cQ)$ and every positive integer~$c$ 
there exists an $F$-hypergraph $(Y, \cR)$
together with a system of copies $\gY\subseteq \binom{(Y, \cR)}{(X, \cQ)}$ such that 
the following statements are true:
\begin{enumerate}[label=\rmlabel]
\item\label{it:ppl1} $\gY\lra (X, \cQ)^F_c$
\item\label{it:ppl2} Whenever $(\til{X}, \til{\cQ})\in\gY$, $\til{F}\in\cR$, and 
		$x$ is a $t$-subset of $V(\til{X})\cap V(\til{F})$, there is some
		$\til{F}'\in \til{\cQ}$ with $x\subseteq V(\til{F}')$.
\end{enumerate}
\end{lem}

\begin{figure}[ht]
\includegraphics[scale=1]{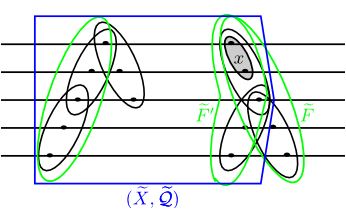}
\caption{Condition~\ref{it:ppl2} from the preliminary partite lemma}
\end{figure}

\begin{proof}
If $\cQ=\vn$ we may just take $(Y, \cR)=(X, \cQ)$ and $\gY=\{(X, \cQ)\}$, and the assertion
holds vacuously. So let us suppose $\cQ\not =\vn$ from now on and set $n=\HJ(|\cQ|, c)$.

Let us write $X=\bl (V^i)_{i=1}^k, E\br$. We begin by defining a $k$-partite
Steiner $(r,t)$-system $Y\in\ccP(k,r,t)$ in the following way: 
\begin{enumerate}
\item[$\bullet$]
For $i\in [k]$ we let 
$V^i(Y)=(V^i)^n$ be the $n$-th Cartesian power of $V^i$. 

A typical vertex from this class
will be written as $\seq{v}^{\, i}=(v^i_1, \ldots, v_n^i)$. For $h\in [n]$ and $i\in [k]$, we 
let $\pi_h \colon V^i(Y) \longrightarrow V^i(X)$ be the natural projection onto the 
$h$-th coordinate defined by $\pi_h(\seq{v}^{\, i})=v^i_h$ for any $\seq{v}^{\, i}$ as in the 
preceding sentence. 
\item[$\bullet$]
A crossing $r$-subset $f$ of $V(Y)$ is declared to be an edge of $Y$ 
if and only if $\pi_h[f]\in E(X)$ holds for all $h\in [n]$, i.e.,
\[
	E(Y)=\Bigl\{f\subseteq \textstyle{\binom{V(Y)}{r}}\,\,\Big|\,\, 
	\text{If } h\in [n], \text{ then } \pi_h[f] \in E(X)\Bigr\}\,.
\]
\end{enumerate}

Let us check that the $k$-partite hypergraph $Y$ thus defined is indeed a Steiner 
$(r, t)$-system.

\begin{claim}
If $f_1, f_2\in E(Y)$ are distinct, then $|f_1\cap f_2| < t$.
\end{claim}

\begin{proof}
Assume for the sake of contradiction that $|f_1\cap f_2| \ge t$.
By projection we get $|\pi_h[f_1] \cap \pi_h[f_2]|\ge t$ for every $h\in [n]$.
Since $X$ is a Steiner $(r, t)$-system, it follows that $\pi_h[f_1] = \pi_h[f_2]$
holds for every $h\in [n]$. This in turn yields $f_1=f_2$.
\end{proof}

The above definition of the edges of $Y$ shows that
\begin{equation} \label{eq:psiY}
	\text{ for every }f\in E(Y) \text{ we have } \psi_Y[f]\in E(F)\,,
\end{equation}
as required by Definition~\ref{dfn:Fhyp}. In what follows it will
be convenient to denote the unique vertex in $V^i$ of a copy 
$\til{F}\in \binom{X}{F}^\tim_{\str}$ by $v^i(\til{F})$, 
so that $V\bl \til{F}\br=\bigl\{v^i\bl \til{F}\br \,\big|\, i\in [k]\bigr\}$
holds for such copies $\til{F}$.

We are now ready to define a collection  
$\cR\subseteq\binom{Y}{F}^\tim_{\str}$. To this end we consider any sequence
$\seq{F}=(F_1, \ldots, F_n)\in\cQ^n$ of not necessarily distinct members of~$\cQ$.
For each $i\in [k]$ we may look at the vertex 
$\seq{v}^{\,i}=\bl v^i(F_1), \ldots, v^i(F_n)\br$ from $V^i(Y)$. Evidently for each 
$e\in E(F)$ we have $\{\seq{v}^{\, i}\,|\,i\in e\}\in E(Y)$. Together with~\eqref{eq:psiY}
this implies that the vertices $\seq{v}^{\,1}, \ldots, \seq{v}^{\,k}$ span a crossing strongly 
induced copy of $F$ in~$Y$, which we will denote by $\lambda(\seq{F})$ in the sequel.
Finally we set 
\[
	\cR=\bigl\{\lambda(\seq{F})\,|\,\seq{F}\in \cQ^n\bigr\}\,.
\]
Observe that 
\[
	\lambda\colon \cQ^n\lra \cR
\]
yields a natural bijective correspondence between the Hales-Jewett 
cube ${\HJC(\cQ, n)=\cQ^n}$ and $\cR$.

Next we address the subhypergraphs of $Y$ corresponding to combinatorial lines. Consider a 
partition $[n]=C\dcup M$ with $M\not =\vn$ as well as a function $g\colon C\lra \cQ$.
Observe that if~$h\in C$, then $g(h)$ is a copy of $F$ in $X$.
To describe the line encoded by the pair $(C, g)$ we will first introduce 
a partite map 
\[
	\varphi_{C, g}\colon V(X)\lra V(Y)
\]
such that for $v\in V^i$ and $h\in [n]$ we have 

\[
	\pi_h\bl\phi_{C, g}(v)\br=
		\begin{cases} 
			v^{i}\bl g(h)\br & \text{ if } h\in C \cr
			v                & \text{ if } h\in M.
		\end{cases}
\]
Let $Z_{C, g}$ be the induced $k$-partite subsystem
of $Y$ spanned by the range of $\phi_{C, g}$.

\begin{claim} \label{clm:line-ind}
If $(C, g)\in\ccL(\cQ, n)$, then $X\cong Z_{C, g} \Str Y$, where 
the isomorphism is given by $\varphi_{C, g}$. 
\end{claim}
 
\begin{proof}
The injectivity of $\varphi_{C, g}$ follows from $M\not =\vn$. 

Next we show that if $e$ is an edge of $X$, then $\phi_{C, g}[e]$ is an edge of $Z_{C, g}$.
Notice that~$\phi_{C, g}[e]$ is a crossing $r$-subset of $Y$. We need to verify that
$\bl \pi_h\circ \phi_{C, g}\br[e]\in E(X)$ holds for all~$h\in [n]$. For $h\in M$ we have
$\bl \pi_h\circ \phi_{C, g}\br[e]=e$. On the other hand, if $h\in C$, then~$g(h)$ is 
some copy $\til{F}\in \binom{X}{F}^\tim_{\str}$. Since we have $\psi_X[e]\in E(F)$
in view of~\eqref{eq:225}, it follows that 
$\bl \pi_h\circ \phi_{C, g}\br[e]\in E\bl \til{F}\br$.

There are two things that remain to be shown, namely that all edges of $Z_{C, g}$ 
have preimages in $X$ and that the inducedness of $Z_{C, g}$ in $Y$ is strong. 
Observe that both of them are implied by the following statement:
\[
	\text{If $f\in E(Y)$ satisfies $|f\cap V(Z_{C, g})|\ge t$, 
	then there is some $e\in E(X)$ with $\phi_{C, g}[e]=f$}\,.
\]

To prove this we put $y=f\cap V(Z_{C, g})$ and let $x\subseteq V(X)$ be the preimage of
$y$ with respect to $\phi_{C, g}$. For every $h\in M$
we have $x\subseteq \pi_h[f]\in E(X)$. Since $X$ is a Steiner $(r, t)$-system
and $|x|\ge t$, it follows that all edges of the form $\pi_h[f]$ with $h\in M$ 
must be the same. In other words there is an edge $e\in E(X)$ with 
$x\subseteq e=\pi_h[f]$ for every $h\in M$. Clearly~$e$ is as desired.
\end{proof}

We keep considering $C$, $M$, and $g$ as above. Let $\eta_{C, g}\colon \cQ\lra \HJC(\cQ, n)$
be the map given by~\eqref{eq:HJ-eta}. Now $\cL_{C, g}=(\lambda\circ\eta_{C, g})[\cQ]$  
is a subset of $\cR$ and one confirms easily that all vertices of the copies of $F$ belonging 
to this set lie in $V(Z_{C, g})$. The next claim asserts that combinatorial lines correspond
to strongly induced copies of $(X, \cQ)$ in $(Y, \cR)$.

\begin{claim} \label{clm:combi-line}
If $(C, g)\in\ccL(\cQ, n)$, then $(X, \cQ)\cong (Z_{C, g}, \cL_{C, g})\Str (Y, \cR)$.
\end{claim}

\begin{proof}
Owing to Claim~\ref{clm:line-ind} the only thing that needs to be checked is 
$\cL_{C, g}=\binom{Z_{C, g}}{F}^\tim_{\str}\cap\cR$. We leave the details to the reader.
\end{proof}

Now we contend that $(Y, \cR)$ and the system of copies
\[
	\gY=\bigl\{(Z_{C, g}, \cL_{C, g})\,|\, (C, g)\in\ccL(\cQ, n)\bigr\}
\]
have the properties~\ref{it:ppl1} and~\ref{it:ppl2} demanded by the preliminary 
partite lemma. 
The first of them is a direct consequence of the Hales-Jewett theorem and 
Claim~\ref{clm:combi-line}. 

Let us now prove~\ref{it:ppl2} for some $(\til{X}, \til{\cQ})\in\gY$, $\til{F}\in \cR$,
and $t$-set $x\subseteq V(\til{X})\cap V(\til{F})$. Choose a partition $[n]=C\dcup M$
with $M\not =\vn$ and a function $g\colon C\lra \cQ$ such that 
$(\til{X}, \til{\cQ})=(Z_{C, g}, \cL_{C, g})$ and let $\til{F}=\lambda(F_1, \ldots, F_n)$. 
The assumption $x\subseteq V(\til{F})$ yields $\pi_h[x]\subseteq V(F_h)$ for every $h\in [n]$.
In order to define $\til{F}'$ we select an arbitrary $h_0\in M$, set
\[
	F_h'=
		\begin{cases}
			F_h & \text{ if } h\in C \cr
			F_{h_0} & \text{ if } h\in M
		\end{cases}
\]
for every $h\in [n]$, and finally we let $\til{F}'=\lambda(F_1', \ldots, F_n')$. 
One sees immediately that $\til{F}'\in\cL_{C, g}$ and $x\subseteq V(\til{F}')$. 
\end{proof}

\subsection{The partite construction}\label{subsec:PC}

We will now provide an abstract description of the partite construction 
that is general enough for our intended applications. Since Steiner $(r, t)$-systems
are in general not closed under the kind of amalgamation we need to perform, we will
explain everything with general $r$-uniform hypergraphs instead. Actually it will be 
among the main difficulties encountered later on to formulate appropriate side conditions 
that allow us to maintain Steinerness throughout the partite construction.

The material that follows splits naturally into four parts. We begin by setting up 
some terminology regarding the ``pictures'' that the partite construction generates.
Next we introduce the so-called ``picture zero'' it is initialised with. In 
Subsubsection~\ref{sssec:amalgam} we discuss the ``amalgamations'' which bring us from one 
picture to the next. Finally we will be in a position to say precisely how the 
partite construction proceeds and what its main partition property (see 
Lemma~\ref{lem:pcrams} below) asserts. 

For the purposes of this subsection, 
we fix an $r$-uniform hypergraph $F$ with $k$ vertices, say. 

\subsubsection{{\bf Pictures}} The pictures we need to deal with for proving
Theorem~\ref{thm:main-a} will be three-layered structures consisting of a partite 
hypergraph, a system of distinguished copies of~$F$, and a system of distinguished
copies of the object for which we intend to find a Ramsey object. As a first step 
towards the definition of these pictures, we talk about hypergraphs with a 
distinguished system of copies of $F$. 

\begin{dfn}
By an {\it $F$-system} we mean a pair $(X, \cQ)$ consisting of 
an $r$-uniform hypergraph $X$ and a collection $\cQ\subseteq\binom{X}{F}$ 
of induced copies of $F$ in $X$. 
\end{dfn}

This concept should not be confused with the $F$-hypergraphs from 
Definition~\ref{dfn:Fhyp}. The differences are that $F$-systems
do not come with a $k$-partite structure and that the hypergraph $X$ 
is not required to be a Steiner $(r, t)$-system for any $t<r$. Accordingly
we cannot demand the copies of $F$ belonging to $\cQ$ to be strongly induced
as we did it in the case of $F$-hypergraphs.

In the sequel we will need to work with two different kinds of $F$-subsystems.

\begin{dfn} \label{dfn:25}
For two $F$-systems $(X, \cQ)$ and~$(Y, \cR)$ we say that the former is a 
{\it semi\-induced subsystem} of the latter if $X\le Y$ 
and~$\cQ\subseteq \cR$ hold. 
If moreover $\cQ= \binom{X}{F}\cap \cR$
we call $(X, \cQ)$ an {\it induced subsystem} of $(Y, \cR)$ and 
write $(X, \cQ)\le (Y, \cR)$.
\end{dfn}

For a collection $\gY$ of semi-induced copies of $(X, \cQ)$
in~$(Y, \cR)$ the partition symbol
\begin{equation}\label{eq:231}
	\gY\lra (X, \cQ)^F_c
\end{equation}
means that for every colouring of $\cR$ with $c$ colours there is some 
$\bl \til{X}, \til{\cQ}\br\in\gY$ for which~$\til{\cQ}$ is monochromatic.

For the remainder of Subsection~\ref{subsec:PC} we fix two $F$-systems $(X, \cQ)$
and $(Y, \cR)$ as well as a system of semi-induced copies $\gY$ such that~\eqref{eq:231} holds.
There arises no loss of generality by assuming 
\[
	V(Y)=[m]\,,
\]
where $m=v_Y$. 
Let us enumerate~$\cR$ and~$\gY$ as 
\begin{equation}\label{eq:242}
	\cR=\bigl\{F_1, \ldots, F_{|\cR|}\bigr\}
\end{equation}
and
\begin{equation}\label{eq:243}
	\gY=\bigl\{\bl X_1, \cQ_1\br, \ldots, \bl X_{|\gY|}, \cQ_{|\gY|}\br\bigr\}\,,
\end{equation}
respectively. 

We may now describe the first two layers of our pictures. 
In the definition that follows, the projection $\psi_{Z}\colon V(Z)\lra [m]$ 
is defined as in~\eqref{eq:psi}.

\begin{figure}[ht]
\includegraphics[scale=.8]{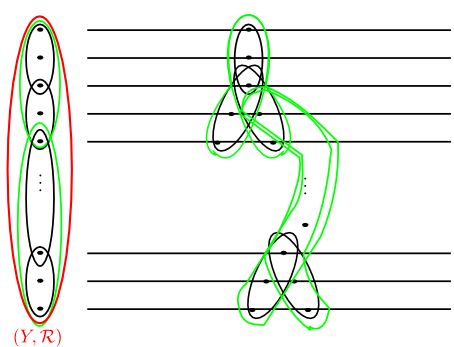}
\caption{A $(Y, \cR)$-hypergraph}
\end{figure}

\begin{dfn} \label{dfn:yr-hg}
A {\it $(Y, \cR)$-hypergraph} is a pair $(Z, \cS)$, 
\begin{enumerate}[label=\rmlabel]
	\item\label{it:yr1} where $Z$ is an $m$-partite $r$-uniform hypergraph,  
	\item\label{it:yr2} such that if $e\in E(Z)$, then $\psi_Z[e]\in E(Y)$
	\item\label{it:yr3} and $\cS$ is a system of crossing induced copies of $F$ in $Z$,
	\item\label{it:yr4} such that if $\til{F}\in\cS$, then $\psi_Z\bl \til{F}\br \in\cR$. 
\end{enumerate}
\end{dfn}

Together with such $(Y, \cR)$-hypergraphs~$(Z, \cS)$, the partite construction will also 
generate systems of so-called good copies of $(X, \cQ)$ within them.
Here a {\it good copy} of $(X, \cQ)$ in~$(Z, \cS)$ is a crossing induced subsystem
$(\widetilde{X}, \widetilde{\cQ})$ of~$(Z, \cS)$ (in the sense of Definition~\ref{dfn:25})
which is isomorphic to a member of $\gY$ via $\psi_Z$. 

\begin{dfn} \label{dfn:pic}
A {\it picture} is a triple $(Z, \cS, \gZ)$ such that   
\begin{enumerate}[label=\rmlabel]
	\item\label{it:pic1} $(Z, \cS)$ is a $(Y, \cR)$-hypergraph 
	\item\label{it:pic2} and $\gZ$ is a system of good copies of $(X, \cQ)$ in $(Z, \cS)$.
\end{enumerate}
If $\Pi=(Z, \cS, \gZ)$ is a picture, we will write $Z(\Pi)=Z$, $V(\Pi)=V(Z)$, $E(\Pi)=E(Z)$, 
$\cS(\Pi)=\cS$, and~${\gZ(\Pi)=\gZ}$. 
\end{dfn}

The next definition clarifies that for two pictures $\Pi$ and $\Pi'$ the notation 
$\Pi\le \Pi'$ has its expected meaning.

\begin{dfn}
For two pictures $\Pi=(Z, \cS, \gZ)$ and $\Pi'=(Z', \cS', \gZ')$ we write $\Pi\le \Pi'$
if the following hold:
\begin{enumerate}[label=\rmlabel]
\item For every $i\in [m]$ we have $V^i(Z)\subseteq V^i(Z')$.
\item $(Z, \cS)$ in an induced $F$-subsystem of $(Z', \cS')$.
\item A copy $\bl\til{X}, \til{\cQ}\br$ of $(X, \cQ)$ with $V\bl \til{X}\br \subseteq V(Z)$
		belongs to $\gZ$ if and only if it belongs to $\gZ'$.
\end{enumerate}
\end{dfn}

\subsubsection{{\bf Picture zero}}
The starting point of the partite construction is a so-called ``picture zero'' 
$\Pi^0=(Z^0, \cS^0, \gZ^0)$, which has associated with each member $(X_y, \cQ_y)$ 
of $\gY$ its own good copy $\bl X^0_y, \cQ^0_y\br$ of $(X, \cQ)$. These good copies 
are to be mutually vertex disjoint. Moreover, for each $y\in [|\gY|]$ the good copy 
$\bl X^0_y, \cQ^0_y\br\in \gZ^0$ is to be placed on the vertex classes
$V^1(Z^0), \ldots, V^m(Z^0)$ of $Z^0$ in such a way that $\big|V(X^0_y) \cap V^i(Z^0)\big|=1$
holds if and only if $i\in V(X_y)$. More exactly, we demand $\bl X^0_y, \cQ^0_y\br$ 
and $(X_y, \cQ_y)$ to be isomorphic via the projection $\psi_{Z^0}$. 

The formal definition that follows summarises this description.    

\begin{dfn} \label{dfn:picz}
{\it Picture zero} is a picture $\Pi^0=(Z^0, \cS^0, \gZ^0)$ with the property 
that we can write 
\[
	\gZ^0=\bigl\{\bl X^0_1, \cQ^0_1\br, \ldots, \bl X^0_{|\gY|}, \cQ^0_{|\gY|}\br\bigr\}
\]
in such a way that the following hold: 
\begin{enumerate}[label=\rmlabel]
	\item\label{it:p01} $V(Z^0)=V(X^0_1) \dcup \ldots \dcup V(X^0_{|\gY|})$ 
	\item\label{it:p02} $E(Z^0)=E(X^0_1) \dcup \ldots \dcup E(X^0_{|\gY|})$
	\item\label{it:p03} $\cS^0= \cQ^0_1 \dcup \ldots \dcup \cQ^0_{|\gY|}$
	\item\label{it:p04} If $y\in [|\gY|]$, then $\bl X^0_y, \cQ^0_y\br$ is 
			isomorphic to $(X_y, \cQ_y)$ via $\psi_{Z^0}$.
\end{enumerate}
\end{dfn}

\begin{figure}[ht]
\includegraphics[scale=1]{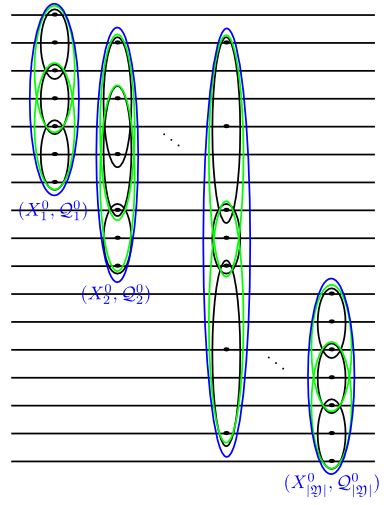}
\caption{Picture zero}
\end{figure}

\subsubsection{{\bf Amalgamation}} \label{sssec:amalgam}
The partite construction itself proceeds in $|\cR|$ successive amalgamation steps. 
To explain what happens in one such step, suppose that we have a picture $\Pi=(Z, \cS, \gZ)$
as well as an integer $\rho\in [|\cR|]$.  

Let $V(F_\rho)=\{j(1), \ldots, j(k)\}$ list the vertices of $F_\rho$ in increasing order. 
Define $(Z_\rho, \cS_\rho)$ to be the $k$-partite $F$-system with 
\begin{align*}
	V^i(Z_\rho)&=V^{j(i)}(Z) \text{ for all } i\in [k]\,,    \\
	E(Z_\rho)  &=\bigl\{e\in E(Z)\,\big|\,\psi_Z[e]\in E(F_\rho)\bigr\}\,,  \\
	\text{ and } \quad \quad \quad \quad \quad 
	\cS_\rho   &=\bigl\{\til{F}\in \cS\,\big|\,\psi_{Z}\bl \til{F}\br=F_\rho\bigr\}\,.
\end{align*}
Observe that $(Z_\rho, \cS_\rho)$ is an induced 
$F$-subsystem of $(Z, \cS)$ due to $F_\rho\le Y$ and the 
conditions~\ref{it:yr2}, \ref{it:yr4} of Definition~\ref{dfn:yr-hg}.

Now suppose $(W, \cP)$ to be a further $k$-partite $F$-system admitting a system 
$\gW$ of $k$-partite induced copies of $(Z_\rho, \cS_\rho)$ with
\[
	\gW\lra (Z_\rho, \cS_\rho)^F_c\,.
\]
Later on we will use either Lemma~\ref{lem:ppl} or Lemma~\ref{lem:cpl} to obtain 
such a system $\gW$.

The {\it amalgamation} we have in mind leads to a new picture 
\[
	\Pi\conc\nolimits_\rho \gW=
	(Z\conc\nolimits_\rho \gW, \cS\conc\nolimits_\rho \gW, \gZ\conc\nolimits_\rho \gW)
\]
the formal definition of which will cover the remainder of this subsubsection.
The underlying idea is that starting from $(W, \cP)$ we extend every copy
$\bl\til{Z}_\rho, \til{\cS}_\rho\br\in\gW$ of $(Z_\rho, \cS_\rho)$ to its own copy 
$\Pi_{\til{Z}_\rho, \til{\cS}_\rho}$
of the picture $\Pi$, keeping these copies of $\Pi$ as disjoint as possible, 
i.e., such that two distinct such copies do only share vertices of $W$ 
with each other. 

\begin{figure}[ht]
\includegraphics[scale=1]{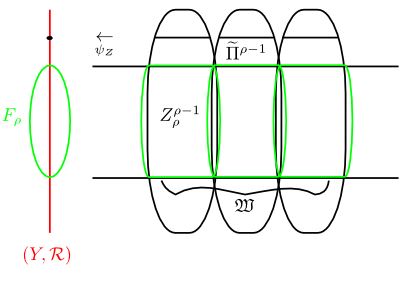}
\caption{Partite amalgamation}
\end{figure}

To begin with, the vertex classes of the desired 
hypergraph $Z\conc_\rho\gW$ are going to be
\[
	V^j(Z\conc\nolimits_\rho\gW)=
			\begin{cases}
			V^i(W) & 
			\text{ if } j=j(i) \text{ holds for some } i\in[k] \cr
			V^j(Z)\times \gW & 
			\text{ if } j\not\in\{j(1), \ldots, j(k)\}\,.
			\end{cases}
\]
Now for every copy $\bl\til{Z}_\rho, \til{\cS}_\rho\br\in \gW$ we fix 
a bijection 
$\overline{\phi}_{\til{Z}_\rho, \til{\cS}_\rho}\colon V(Z_\rho)\lra V(\til{Z}_\rho)$ 
establishing an isomorphism between $(Z_\rho, \cS_\rho)$ and $\bl\til{Z}_\rho, \til{\cS}_\rho\br$ 
respecting the $k$-partite structure and then we extend 
$\overline{\phi}_{\til{Z}_\rho, \til{\cS}_\rho}$ to an injective map
$\phi_{\til{Z}_\rho, \til{\cS}_\rho}\colon V(Z)\lra V(Z\conc\gW)$ given by
\[
	\phi_{\til{Z}_\rho, \til{\cS}_\rho}(v)=
				\begin{cases}
				\overline{\phi}_{\til{Z}_\rho, \til{\cS}_\rho}(v) &
				\text{ if } v\in V(Z_\rho) \cr
								\bl v, \bl\til{Z}_\rho, \til{\cS}_\rho\br\br &
				\text{ otherwise}.
								\end{cases}
\]
Further, we let the picture $\Pi_{\til{Z}_\rho, \til{\cS}_\rho}$ be the image of
$\Pi$ under $\phi_{\til{Z}_\rho, \til{\cS}_\rho}$. 

Finally $\Pi\conc\nolimits_\rho \gW=(Z\conc\nolimits_\rho \gW, 
\cS\conc\nolimits_\rho \gW, \gZ\conc\nolimits_\rho \gW)$ 
is defined to be the union of all these pictures
$\Pi_{\til{Z}_\rho, \til{\cS}_\rho}$ as $\bl \til{Z}_\rho, \til{\cS}_\rho \br$
varies over $\gW$, so explicitly we stipulate
\begin{align*}
	E(Z\conc\nolimits_\rho\gW) &=\bigcup\,\bigl\{ E\bl \Pi_{\til{Z}_\rho, \til{\cS}_\rho}\br
	\,\big|\, \bl \til{Z}_\rho, \til{\cS}_\rho\br \in \gW\bigr\}\,,  \\
	\cS\conc\nolimits_\rho\gW &=\bigcup\,\bigl\{ \cS\bl \Pi_{\til{Z}_\rho, \til{\cS}_\rho}\br
	\,\big|\, \bl \til{Z}_\rho, \til{\cS}_\rho\br \in \gW\bigr\}\,,  \\
    \text{ and } \qquad
	\gZ\conc\nolimits_\rho\gW &=\bigcup\,\bigl\{ \gZ\bl \Pi_{\til{Z}_\rho, \til{\cS}_\rho}\br
	\,\big|\, \bl \til{Z}_\rho, \til{\cS}_\rho
	\br \in \gW\bigr\}\,.
\end{align*}

It is easy to check that $\Pi\conc\nolimits_\rho \gW$ is again a picture and that
we have 
\begin{equation} \label{eq:part-ind}
	\Pi_{\til{Z}_\rho, \til{\cS}_\rho} \le \Pi\conc\nolimits_\rho \gW
	\quad \text{ for every }
	\bl \til{Z}_\rho, \til{\cS}_\rho\br \in \gW\,.
\end{equation} 
The pictures of the form $\Pi_{\til{Z}_\rho, \til{\cS}_\rho}$ will be referred to as the
{\it canonical copies} of $\Pi$ in $\Pi\conc\nolimits_\rho \gW$.

\subsubsection{{\bf The final picture}}
Having thus defined amalgamations, we may proceed by describing the partite construction
itself. 

\begin{dfn}
Suppose that we construct a sequence 
\[
	\Pi^0=(Z^0, \cS^0, \gZ^0), \ldots, \Pi^{|\cR|}=\bl Z^{|\cR|}, \cS^{|\cR|}, \gZ^{|\cR|}\br 
\]
of pictures starting with picture zero and 
such that for each $\rho\in [|\cR|]$ we have 
\[
	\Pi^\rho=\Pi^{\rho-1}\conc\nolimits_\rho \gW^\rho, 
	\quad \text{where} \quad  \gW^\rho\lra (Z^{\rho-1}_\rho, \cS^{\rho-1}_\rho)^F_c\,.
\]
Then we say that the picture $\Pi^{|\cR|}$ has arisen by means of a 
{\it partite construction} and about~$(Y, \cR)$ itself we say that it 
was put {\it senkrecht}.
\end{dfn}
 
We conclude this subsection by stating an important property of this construction. 

\begin{lem} \label{lem:pcrams}
If the picture $\Pi=(Z, \cS, \gZ)$ arises by a partite construction, then
\[
	\gZ\lra (X, \cQ)^F_c\,.
\]
\end{lem}

\begin{proof}
Keeping the above notation in force we have $Z=Z^{|\cR|}$, $\cS=\cS^{|\cR|}$, and 
${\gZ=\gZ^{|\cR|}}$. Consider any colouring $\gamma\colon \cS\lra [c]$. 
When we go backwards through the partite construction, the first partition property 
we may invoke is $\gW^{|\cR|}\lra \bl Z^{|\cR|-1}_{|\cR|}, \cS^{|\cR|-1}_{|\cR|}\br^F_c$. 
It leads to a canonical copy $\til{\Pi}^{|\cR|-1}$ of $\Pi^{|\cR|-1}$ 
and to a colour $\phi(|\cR|)\in [c]$ such that $\gamma(\til{F})=\phi(|\cR|)$ holds 
for all $\til{F}\in \til{\cS}^{|\cR|-1}$ with $\psi_Z(\til{F})=F_{|\cR|}$. 

Iterating this argument $|\cR|-1$ further times we may use the partition properties 
of~$\gW^{|\cR|-1}, \ldots, \gW^1$ in turn and ultimately obtain a function 
$\phi\colon [|\cR|]\lra [c]$ as well as a copy 
$\til{\Pi}^0=\bl \til{Z}^0, \til{\cS}^0, \til{\gZ}^0\br\le \Pi^{|\cR|}$
of picture zero such that
\begin{equation}\label{eq:phidef}
	\text{if }
	\til{F}\in \til{\cS}^0 
	\text{ and }
	\psi_Z(\til{F})=F_\rho, 
	\text{ then }
	\gamma(\til{F})=\phi(\rho)\,.
\end{equation} 

Next we apply~\eqref{eq:231} to the colouring $F_\rho\longmapsto \phi(\rho)$ of $\cR$
and get some $y\in [|\gY|]$ such that~$\cQ_y$ is monochromatic under this colouring.
Due to~\eqref{eq:phidef} this means that the good copy 
$(\til{X}_y, \til{\cQ}_y)\in \til{\gZ}^0$ of 
$(X, \cQ)$ corresponding to $(X^0_y, \cQ^0_y)\in \gZ^0$ has 
$\til{\cQ}_y$ monochromatic under~$\gamma$. 
\end{proof}

\subsection{The clean partite lemma} \label{subsec:PL}

Now we analyse what happens when we clean the preliminary partite lemma 
by means of the partite method. Resuming the discussion of Subsection~\ref{subsec:PPL},
we suppose again that $F$ is a Steiner $(r, t)$-system with $V(F)=[k]$.

\begin{lem}[Clean partite lemma] \label{lem:cpl}
Given an $F$-hypergraph $(X, \cQ)$ and an integer~$c$ there is an $F$-hypergraph $(Z, \cS)$
and a system of copies $\gZ\subseteq \binom{(Z, \cS)}{(X, \cQ)}$ such that 
we have:
	\begin{enumerate}[label=\rmlabel]
		\item\label{it:cpl1} $\gZ\lra (X, \cQ)^F_c$
		\item\label{it:cpl2} If $(X', \cQ')$ and $(X'', \cQ'')$ are distinct members of 
		$\gZ$ and $x\subseteq V(X')\cap V(X'')$ has size $t$, then there are $F'\in \cQ'$ 
		and $F''\in \cQ''$ with $x\subseteq V(F')\cap V(F'')$.
\end{enumerate}
\end{lem}

\begin{proof}
Owing to the preliminary partite lemma, there exists an $F$-hypergraph $(Y, \cR)$ 
and a system of copies $\gY\subseteq \binom{(Y, \cR)}{(X, \cQ)}$ with
$\gY\lra (X, \cQ)^F_c$
such that 
\begin{equation} \label{eq:prelimprop}
	\forall \bl\til{X}, \til{\cQ}\br\in\gY \,\, 
	\forall \til{F}\in\cR \,\,
	\forall x\subseteq V\bl\til{X}\br\cap V\bl\til{F}\br\,\,
		\bigl[
				|x|=t \, \Longrightarrow \,
				\exists \til{F}'\in \til{\cQ} \,\colon x\subseteq V\bl\til{F}'\br
		\bigr]\,.
\end{equation}
As this situation does not change by adding isolated vertices to $Y$ we may suppose 
for notational simplicity that this $k$-partite Steiner system is balanced, i.e.,
that 
\[
	|V^1(Y)|=\ldots=|V^{k}(Y)|=\tfrac mk
\]
holds for some positive multiple $m$ of $k$. 
Moreover, we may relabel the vertices of $Y$ so as to obtain 
\begin{equation} \label{eq:viy}
	V^i(Y)=\bigl[\tfrac {(i-1)m}k +1, \tfrac {im}k \bigr]
	\quad \text{for all } i\in [k]
\end{equation}
and, hence, $V(Y)=[m]$. 

Now we put $(Y, \cR)$ senkrecht and run the partite construction. 
Reusing the notation of Subsection~\ref{subsec:PC} we enumerate
$\cR$ and $\gY$ as in~\eqref{eq:242} and~\eqref{eq:243}, and let 
$\Pi^0=(Z^0, \cS^0, \gZ^0)$ be picture zero as described in 
Definition~\ref{dfn:picz}.
The goal is to construct recursively a sequence
$\Pi^1=(Z^1, \cS^1, \gZ^1), \ldots, \Pi^{|\cR|}=(Z^{|\cR|}, \cS^{|\cR|}, \gZ^{|\cR|})$ 
of pictures with $\Pi^\rho=\Pi^{\rho-1}\conc_\rho \gW^\rho$
for each $\rho\in [|\cR|]$, where~$\gW^\rho$ should be some Ramsey system with
$\gW^\rho\lra (Z^{\rho-1}_\rho, \cS^{\rho-1}_\rho)^F_c$. 

We intend to maintain throughout the construction   
\begin{enumerate}[label=\glabel]
\item\label{it:alpha}
that each $Z^\rho$ is a Steiner $(r, t)$-system 
\item\label{it:beta}
and that the copies of $F$ belonging to $\cS^\rho$ 
are strongly induced in $Z^\rho$. 
\end{enumerate}

Notice that picture zero has these properties because
$X$ is a Steiner $(r, t)$-system and the copies $\til{F}\in\cQ$ are strongly 
induced in $X$. Now suppose that for some $\rho\in [\cR]$ we have already managed to
construct a picture $\Pi^{\rho-1}=(Z^{\rho-1}, \cS^{\rho-1}, \gZ^{\rho-1})$ 
satisfying~\ref{it:alpha} and~\ref{it:beta}. Then 
$\bl Z^{\rho-1}_\rho, \cS^{\rho-1}_\rho\br$ is an $F$-hypergraph 
and the preliminary partite lemma allows us to choose $\gW^\rho$ in such a way that
$\bl Z^\rho_\rho, \cS^\rho_\rho\br$ is an $F$-hypergraph and that
the copies $\bl \til{Z}^{\rho-1}_\rho, \til{\cS}^{\rho-1}_\rho\br\in\gW^\rho$ are strongly
induced in $\bl Z^\rho_\rho, \cS^\rho_\rho\br$. It remains to be checked that the picture
$\Pi^\rho=\Pi^{\rho-1}\conc_\rho \gW^\rho$ has the properties~\ref{it:alpha} and~\ref{it:beta}
as well. 

Starting with~\ref{it:alpha} we consider any two edges $e$ and $e'$ of $Z^\rho$ 
with $|e\cap e'|\ge t$. We are to prove that $e=e'$. 
If $e\cap e'\not\subseteq V(Z^{\rho}_\rho)$ there is a single canonical copy 
$\til{\Pi}^{\rho-1}$ of the previous picture containing both $e$ and $e'$,
meaning that~\eqref{eq:part-ind} and the induction hypothesis lead to the desired conclusion. 
So we may suppose $e\cap e'\subseteq V(Z^{\rho}_\rho)$ from now on, whence 
$\psi_{Z^\rho}[e\cap e']\subseteq V(F_\rho)$. 
Using $|e\cap e'|\ge t$ and $F_\rho\Str Y$
we deduce~$e, e'\subseteq V(Z^{\rho}_\rho)$ and, as $Z^{\rho}_\rho$ 
is a Steiner $(r, t)$-system, this entails indeed $e=e'$. 

In view of 
$\cR\subseteq\binom{Y}{F}_{\str}$
and Definition~\ref{dfn:yr-hg}\ref{it:yr4} picture $\Pi^\rho$ satisfies~\ref{it:beta} as well.
This completes the proof that the partite construction we were aiming at can indeed be 
carried out.

Since the members of $\gY$ are strongly induced in $(Y, \cR)$, a similar argument yields
\begin{equation}\label{eq:gZ}\til{X}\Str Z^{|\cR|}
\text{ for all }
\bl \til{X}, \til{\cQ}\br \in \gZ^{|\cR|}\,.
\end{equation}
The $F$-hypergraph $(Z, \cS)$ and the system $\gZ$ promised by the clean partite lemma 
will essentially be a $k$-partite reorganisation of the last picture $\Pi^{|\cR|}$
and the partition property~\ref{it:cpl1} will be an easy consequence of Lemma~\ref{lem:pcrams}.
To get~\ref{it:cpl2} as well we prove that all the pictures we have generated satisfy 
this intersection property. In other words we contend that for  every nonnegative 
$\rho\le |\cR|$ we have:

\smallskip

{\it \hskip.3cm $(*)_\rho$ \hskip.2cm If $(X', \cQ'), (X'', \cQ'')\in \gZ^\rho$ are distinct 
			and 
			$x\subseteq V(X')\cap V(X'')$ is a $t$-set, then there are 
			$F'\in \cQ'$ and $F''\in \cQ''$ with $x\subseteq V(F')\cap V(F'')$.
}

\smallskip

Let us prove this by induction on $\rho$. The base case $\rho=0$ is clear because by 
Definition~\ref{dfn:picz}\ref{it:p01} there are no distinct members 
$(X', \cQ'), (X'', \cQ'')\in \gZ^0$ with $V(X')\cap V(X'')\not =\vn$.

For the inductive step we suppose that $(*)_{\rho-1}$ holds for some $\rho\in[|\cR|]$
and that $(X', \cQ')$, $(X'', \cQ'')$ as well as $x$ are as above.
Let 
$\til{\Pi}^{\rho-1}$ and $\wh{\Pi}^{\rho-1}$ 
be the canonical copies of picture~$\Pi^{\rho-1}$ with 
$(X', \cQ')\in \gZ\bl\til{\Pi}^{\rho-1}\br$ and $(X'', \cQ'')\in \gZ\bl\wh{\Pi}^{\rho-1}\br$.
If 
$\til{\Pi}^{\rho-1}=\wh{\Pi}^{\rho-1}$
the desired con\-clusion can be drawn from the induction hypothesis.

Otherwise $\psi_{Z^\rho}$ projects $x$ to a $t$-subset of~$V(F_\rho)$
and $(X', \cQ')$ onto some member of~$\gY$, say $(X_y, \cQ_y)$.
Applying~\eqref{eq:prelimprop} to $(X_y, \cQ_y)$, $F_\rho$, and~$\psi_{Z^{\rho}}[x]$
we get some $\til{F}'\in \cQ_y$ with $\psi_{Z^{\rho}}[x]\subseteq V(\til{F}')$.
The member $F'$ of $\cQ'$ that $\psi_{Z^{\rho}}$ projects to $\til{F}'$ satisfies
$x\subseteq V(F')$ and for similar reasons there is some $F''\in \cQ''$ with
$x\subseteq V(F'')$. This completes the inductive step.

Finally we put everything together:
since $(Z^{|\cR|}, \cS^{|\cR|})$ is a $(Y, \cR)$-hypergraph,~\eqref{eq:viy} tells us that 
there is a $k$-partite Steiner $(r, t)$-system $Z$ with
\[
	V^i(Z)= V^{(i-1)m/k+1}\bl Z^{|\cR|}\br \mydcup \ldots \mydcup V^{im/k}\bl Z^{|\cR|}\br
\]
for all $i\in [k]$ and $E(Z)=E(Z^{|\cR|})$. So informally $Z$ is the same 
as $Z^{|\cR|}$ except for having a different partite structure. 
Setting $\cS=\cS^{|\cR|}$ and $\gZ=\gZ^{|\cR|}$,
the same argument shows that $(Z, \cS)$ is actually an $F$-hypergraph
and that we have $\gZ\subseteq \binom{(Z, \cS)}{(X, \cQ)}$ by~\eqref{eq:gZ}.

Now $(Z, \cS)$ and $\gZ$ have the desired properties, because~\ref{it:cpl1} follows from
Lemma~\ref{lem:pcrams} and~\ref{it:cpl2} holds in view of $(*)_{|\cR|}$. 
\end{proof}

\subsection{The proof of Theorem~\ref{thm:main-a}} \label{subsec:APC}

Now we are ready to prove that $\cS^{\,\str}_<(r, t)$ is a Ramsey class.
To this end let any $F_<, X_<\in \cS^{\,\str}_<(r, t)$ be given. 
We are to find some ${Z_<\in \cS^{\,\str}_<(r, t)}$ with 
$Z_< \overset{\hskip-.25em\str}{\lra} (X_<)^{F_<}_c$, where the black triangle 
above the partition arrow is supposed to remind us that we are aiming for 
a {\it strongly} induced copy to be monochromatic. 
Recall that by Theorem~\ref{thm:hrc} there is an ordered $r$-uniform hypergraph~$Y_<$ 
with 
\begin{equation}\label{eq:yrams}
	Y_<\lra (X_<)^{F_<}_c\,,
\end{equation}
but this is not enough because it is neither clear whether~$Y_<$ is a Steiner 
$(r, t)$-system nor, if it actually is, whether the monochromatic copy of $X_<$ it leads to 
could always be chosen to be strongly induced. To overcome these problems we define 
$\cR$ to be the subset of $\binom{Y}{F}$ corresponding to $\binom{Y_<}{F_<}$, 
put $(Y, \cR)$ senkrecht, and try to run the partite construction
using the clean partite lemma in every amalgamation step.

For this purpose we suppose that $V(F_<)=[k]$ and $V(Y_<)=[m]$ hold for $k=v_F$ and~$m=v_Y$,
and that the orderings of $F_<$ and $Y_<$ agree with the natural orderings of $[k]$ and~$[m]$,
respectively. Moreover we let $\cQ$ be the subset of 
$\binom{X}{F}_{\str}$ corresponding to $\binom{X_<}{F_<}_{\str}$. By~\eqref{eq:yrams} there
exists a system $\gY$ of semi-induced copies of $(X, \cQ)$ in $(Y, \cR)$ with 
\[
	\gY\longrightarrow (X, \cQ)^F_c\,.
\]
Enumerate $\cR$ and $\gY$ as in~\eqref{eq:242} and~\eqref{eq:243}, and let the picture zero 
corresponding to this situation be given as in Definition~\ref{dfn:picz}
by $\Pi^0=(Z^0, \cS^0, \gZ^0)$. We intend to run the partite construction, thus generating
a sequence 
$\Pi^0=(Z^0, \cS^0, \gZ^0), \ldots, \Pi^{|\cR|}=\bl Z^{|\cR|}, \cS^{|\cR|}, \gZ^{|\cR|}\br$
of pictures.

Suppose that for some $\rho\in [|\cR|]$ we have already managed to obtain, 
after $\rho-1$ steps, the picture $\Pi^{\rho-1}=(Z^{\rho-1}, \cS^{\rho-1}, \gZ^{\rho-1})$
such that the following conditions hold:
\begin{enumerate}[leftmargin=5em] 
	\item[$(a)_{\rho-1}\,\,$] The hypergraph $Z^{\rho-1}$ is a Steiner $(r, t)$-system. 
	\item[$(b)_{\rho-1}\,\,$] Every member of $\cS^{\rho-1}$ is strongly induced in $Z^{\rho-1}$.
	\item[$(c)_{\rho-1}\,\,$] If $(\til{X}, \til{\cQ})\in \gZ^{\rho-1}$, then $\til{X}$ 
					is strongly induced in $Z^{\rho-1}$.
\end{enumerate}

Observe that these are reasonable assumptions, since picture zero evidently 
satisfies $(a)_0$, $(b)_0$, and $(c)_0$.

Now, in particular, the $k$-partite $F$-system $(Z^{\rho-1}_\rho, \cS^{\rho-1}_\rho)$ 
is an $F$-hypergraph by $(a)_{\rho-1}$ and $(b)_{\rho-1}$. Owing to the clean partite
lemma there exists an $F$-hypergraph $(Z^{\rho}_\rho, \cS^{\rho}_\rho)$ together with a 
system $\gW^\rho$ of partite, strongly induced copies of $(Z^{\rho-1}_\rho, \cS^{\rho-1}_\rho)$
with ${\gW^\rho\lra (Z^{\rho-1}_\rho, \cS^{\rho-1}_\rho)}$ 
such that 
\begin{enumerate} 
	\item[$(*)_\rho$] If $\bl \til{Z}^{\rho-1}_\rho, \til{\cS}^{\rho-1}_\rho\br, 
			\bl \wh{Z}^{\rho-1}_\rho, \wh{\cS}^{\rho-1}_\rho\br\in \gW^\rho$
			are distinct and 
			$x\subseteq V\bl \til{Z}^{\rho-1}_\rho\br\cap V\bl \wh{Z}^{\rho-1}_\rho\br$
			is a~$t$-set, then there are $\til{F}\in \til{\cS}^{\rho-1}_\rho$ and
			$\wh{F}\in \wh{\cS}^{\rho-1}_\rho$ with 
			$x\subseteq V\bl \til{F} \br \cap V\bl \wh{F} \br$.
\end{enumerate} 

We define $\Pi^\rho=(Z^\rho, \cS^\rho, \gZ^\rho)=\Pi^{\rho-1}\conc_\rho \gW^\rho$ and contend
that this picture satisfies $(a)_\rho$,~$(b)_\rho$, and $(c)_\rho$. It should be clear that
these claims easily follow from the induction hypothesis and the following statement:
\begin{equation}\label{eq:picinduced}
	\text{If }
	\til{\Pi}^{\rho-1}
	\text{ is a canonical copy of }
	\Pi^{\rho-1}, 
	\text{ then } 
	Z\bl \til{\Pi}^{\rho-1}\br \Str Z(\Pi^\rho)\,. 
\end{equation}
Before proving this, we record two other properties of our construction.
First, the copies from $\gW^\rho$ being strongly induced in $(Z^{\rho}_\rho, \cS^{\rho}_\rho)$
entails:  
\begin{enumerate} 
	\item[$(\boxplus)_\rho$]	
	If $e\in E\bl Z^\rho_\rho\br$ and if $\til{\Pi}^{\rho-1}$
	is a canonical copy of $\Pi^{\rho-1}$ with
	$\big|e\cap V\bl \til{\Pi}^{\rho-1} \br\big|\ge t$,
	then~$e\in E\bl \til{\Pi}^{\rho-1}\br$. 
\end{enumerate}
Second we reformulate $(*)_\rho$ in a more picturesque way:

\begin{enumerate} 
	\item[$(\boxtimes)_\rho$] If $\til{\Pi}^{\rho-1}, \wh{\Pi}^{\rho-1}$ are distinct 
			canonical copies of $\Pi^{\rho-1}$ in $\Pi^\rho$ and
			$x\subseteq V\bl \til{\Pi}^{\rho-1}\br\cap V\bl \wh{\Pi}^{\rho-1}\br$
			is a~$t$-set, then there are 
			$\til{F}\in \cS\bl \til{\Pi}^{\rho-1}\br$ and
			$\wh{F}\in \cS\bl \wh{\Pi}^{\rho-1}\br$ with 
			$x\subseteq V\bl \til{F} \br \cap V\bl \wh{F} \br$
			and $\til{F}, \wh{F}\Str Z^\rho_\rho$.
\end{enumerate} 

Now we are ready to confirm~\eqref{eq:picinduced}. To this end, let a canonical copy
$\til{\Pi}^{\rho-1}$, an edge~$e\in E(\Pi^\rho)$ and a $t$-set 
$x\subseteq e\cap V\bl \til{\Pi}^{\rho-1}\br$ be given. We are to prove that 
$e\in E\bl \til{\Pi}^{\rho-1}\br$. Let $\wh{\Pi}^{\rho-1}$ be the canonical copy
of $\Pi^{\rho-1}$ with $e\in E\bl \wh{\Pi}^{\rho-1}\br$. Since we are otherwise done, 
we may suppose that $\til{\Pi}^{\rho-1}$ and $\wh{\Pi}^{\rho-1}$ are distinct.
Owing to $(\boxtimes)_\rho$ there exists some $\wh{F}\in \cS\bl \wh{\Pi}^{\rho-1}\br$
with $x\subseteq V\bl \wh{F} \br$ and $\wh{F}\Str Z^\rho_\rho$. 
By $(b)_{\rho-1}$ we also have $\wh{F} \Str Z\bl \wh{\Pi}^{\rho-1}\br$.
Together with $x\subseteq V\bl\wh{F}\br\cap e$ this implies
$e\in E\bl\wh{F}\br\subseteq E\bl Z^\rho_\rho\br$. Thus $(\boxplus)_\rho$ yields
~$e\in E\bl \til{\Pi}^{\rho-1}\br$ and consequently $Z\bl \til{\Pi}^{\rho-1}\br$
is indeed strongly induced $Z(\Pi^\rho)$. We have thereby completed the
proof of~\eqref{eq:picinduced} and, hence, the proof of $(a)_\rho$,~$(b)_\rho$, and $(c)_\rho$.

We have thereby shown that the envisaged partite construction can indeed be carried out.
Recall that Lemma~\ref{lem:pcrams} gives 
\begin{equation} \label{eq:final}
	\gZ^{|\cR|}\lra (X, \cQ)^F_c\,.
\end{equation}

Now let $Z_<$ be the ordered Steiner $(r, t)$-system obtained from $Z^{|\cR|}$ 
by ordering the vertices in any way satisfying 
\[
	V^1\bl Z^{|\cR|}\br<\ldots <V^M\bl Z^{|\cR|}\br
\]
and forgetting the partite structure. Then we have $Z_<\in \ccS^{\,\str}_<(r, t)$
by $(a)_{|\cR|}$. Moreover it is easy to deduce from $(b)_{|\cR|}$, $(c)_{|\cR|}$,
and~\eqref{eq:final} that $Z_< \overset{\hskip-.25em\str}{\lra} (X_<)^{F_<}_c$ holds 
in the sense of $\ccS^{\,\str}_<(r, t)$. 
This means that $Z_<$ has the desired Ramsey property and
the proof of Theorem~\ref{thm:main-a} is complete.

\section{The other classes} \label{sec:neg}

In this section we shall prove Theorem~\ref{thm:main-b}. To this end we need to show, 
on the one hand, that under certain conditions a class $\ccT(r, t)$ has the $F$-Ramsey 
property and, on the other hand, that there is a counterexample if these conditions fail. 
We call the results of the former type {\it positive} and those of the 
latter type {\it negative}.
  
\subsection{Positive results} 
The positive part of Theorem~\ref{thm:main-b} is actually a direct corollary of 
Theorem~\ref{thm:main-a}. The extra assumptions~\ref{it:main-uno} and \ref{it:main-weak} 
from Theorem~\ref{thm:main-b} are in the following way helpful for seeing this:

\begin{enumerate}
\item[$\bullet$] If $F\in\ccS(r, t)$ is homogeneous, then there is a unique ordered 
		version $F_<$ of $F$ and for every $G{_<\in\ccS_<(r, t)}$ there is natural 
		bijective correspondence between $\binom{G_<}{F_<}$ and~$\binom{G}{F}$. 
\item[$\bullet$] If $r=t$ or if $F$ is a complete Steiner $(r, t)$-system, 
		then every induced copy of $F$ is in fact strongly induced.
\end{enumerate}

One way to show that these observations and Theorem~\ref{thm:main-a} imply the positive part 
of Theorem~\ref{thm:main-b} is to look separately at the three cases 
$\ccT=\ccS$, $\ccT=\ccS_<$, and~$\ccT=\ccS^{\,\str}$. (Recall that the 
case $\ccT=\ccS^{\,\str}_<$ was already covered by Theorem~\ref{thm:main-a}).

Let us illustrate this by treating the class~$\ccS(r, t)$. 
Suppose to this end that $F, G\in \ccS(r, t)$ are given, where $F$ obeys 
conditions~\ref{it:main-uno} and \ref{it:main-weak}. That is, $F$ is 
homogeneous and, in case $t<r$, it is also a complete Steiner $(r, t)$-system. 
Take any ordered versions~$F_<$ and~$G_<$ of~$F$ and~$G$,
and let $H_<\in\ccS^{\,\str}_<(r, t)$ with $H_<\lra (G_<)^{F_<}_c$ be given by 
Theorem~\ref{thm:main-a}. It suffices to confirm that $H$ is as desired. So consider any
$c$-colouring of $\binom{H}{F}$. Of course, this induces a $c$-colouring of 
$\binom{H_<}{F_<}_{\str}$ and by our choice of $H$ there is a 
strongly induced copy~$\til{G}_<$ of~$G_<$ in~$H_<$ for which~$\binom{\til{G}_<}{F_<}_{\str}$
is monochromatic. Now $\til{G}$ is, in particular, an induced copy of~$G$ in~$H$ and 
by the observations above $\binom{G}{F}$ is monochromatic with respect to the colouring we 
started with. 

Similar but marginally easier considerations apply to the classes 
$\ccS_<(r, t)$ and $\ccS^{\,\str}(r, t)$ as well. We leave the details to the reader. 
     
\subsection{Negative results} The proofs of most of our negative results
will utilise the following lemma which closely follows the lines 
of~\cite{NeRo2}*{Theorem~2}.

\begin{lem}\label{lem:order}
For every $K\in\ccS(r,t)$ there is some $G\in\ccS(r,t)$ such that no matter how we
assign orderings to the vertex sets of $G$ and $K$, thus obtaining $G_<, K_<\in\ccS_<(r,t)$,
we will always have $\binom{G_<}{K_<}_{\str}\not =\varnothing$.
\end{lem}

\begin{proof} Due to the similarity with the argument from~\cite{NeRo2} 
we only give a sketch. If $K$ is homogeneous we may take $G=K$, so suppose from now 
on that the number $m$ of linear orderings on $V(K)$ that lead to different ordered Steiner 
systems is greater than $1$. If the integer $N$ is large enough depending on $v_K$, then an 
easy probabilistic argument shows that there is some $H\in \ccS(v_K,2)$ with $v_H=N$ and 
$e_H=\Omega(N^2)$. Let $H_<$ be any ordered version of~$H$ and let $G_<\in\ccS(r,t)$ be the 
random ordered Steiner system obtained by inserting independently and uniformly at random 
one of the $m$ ordered versions of $K$ into each edge of $H$. 
If $N$ and $H$ were chosen so large that
\[
	m\cdot N!\bigl(\tfrac{m-1}m\bigr)^{e_H}<1\,,
\]
then with positive probability $G$ will be as desired.
\end{proof}

\begin{rem}
Lemma~\ref{lem:order} says that the class $\cS^{\,\str}(r, t)$ has the so-called
{\it ordering property}. There is an alternative proof of this fact using 
Theorem~\ref{thm:main-a}. For a similar argument we refer to~\cite{NeRo75}.
\end{rem}

\subsubsection{{\bf Unordered classes and homogeneity}}
In this subsubsection we show that condition~\ref{it:main-uno} from Theorem~\ref{thm:main-b}
is indeed necessary (see Corollary~\ref{cor:33} below). 
The next result states slightly more than what we need.

\begin{prop}
For every $F\in\ccS(r,t)$ that is not homogeneous there is some $G\in\ccS(r,t)$ such that
for every $H\in\ccS(r,t)$ there is a red-blue colouring of $\binom{H}{F}$ such that
for no $\widetilde{G}\in\binom{H}{G}$ the set $\binom{\widetilde{G}}{F}_{\str}$ is 
monochromatic.
\end{prop}  

\begin{proof}
Let $K\in\ccS(r,t)$ be the disjoint union of two copies of $F$ and let $G\in\ccS(r,t)$
be obtained by applying the previous lemma to $K$. We contend that $G$ has the requested 
property. 

To confirm this we consider any two distinct ordered versions of $F$, say $F'_<$ and
$F''_<$. Let $H\in\ccS(r,t)$ be arbitrary and let $H_<$ be any ordering of $H$.
Now colour all members of $\binom{H}{F}$ that are isomorphic to $F'_<$ under this ordering 
{\it red}, those isomorphic to $F''_<$ {\it blue}, and the remaining ones arbitrarily either 
red or blue. 

Now look at any $\widetilde{G}\in\binom{H}{G}$. Notice that 
$\widetilde{G}$ inherits an ordering from $H_<$, thus becoming some 
$\widetilde{G}_<\in\ccS_<(r,t)$. Let $K_<$ be obtained from $K$ by ordering its vertices 
in such a way that its first $v_F$ vertices form a copy of $F'_<$ while its remaining
vertices from a copy of $F''_<$. By the choice of $G$ there is a strongly induced copy 
$\widetilde{K}_<$ of $K_<$ in $\widetilde{G}_<$. By the construction of~$K_<$ the set
$\binom{\widetilde{K}}{F}_{\str}$ contains copies of both colours and, hence, so 
does~$\binom{\widetilde{G}}{F}_{\str}$.
\end{proof}
   
\begin{cor}\label{cor:33}
If $F\in \ccS(r,t)$ is not homogeneous, then neither of the two unordered classes of Steiner 
$(r,t)$-systems has the $F$-Ramsey property. \hfill $\Box$
\end{cor}   

\subsubsection{{\bf Weak classes and completeness}} \label{subsec:weak} 
Finally we need to show that condition~\ref{it:main-weak} from Theorem~\ref{thm:main-b}
is necessary. It seems convenient to deal with the cases $\ccT=\ccS_<$ and $\ccT=\ccS$
separately. We begin with the easier, ordered case.

\begin{lem} \label{lem:34}
If $t<r$ and $F_<\in \ccS_<(r,t)$ is not complete, then $\ccS_<(r,t)$ does not have the
$F_<$-Ramsey property.
\end{lem}

\begin{proof}
Since $F_<$ is not complete, there is a $t$-set $x\subseteq V(F_<)$ that does not appear 
in any edge of $F_<$. We want to form new ordered Steiner systems by adding an edge 
to $F_<$ through~$x$ and $r-t$ new vertices. In view of $t<r$, this can be done in at least 
two nonisomorphic ways. Let $F'_<$ and $F''_<$ be two distinct ordered Steiner $(r, t)$-systems
that can arise in this way and let $G_<$ be their disjoint union. 

Now if any $H_<\in\ccS_<(r,t)$ is given, we may colour $\binom{H_<}{F_<}$ in such a way that 
exactly those copies of $F_<$ that sit in copies of $F'_<$ are red whilst all others are blue.
As the red copies of~$F_<$ cannot sit in copies of $F''_<$ as well, there is no 
monochromatic copy of $G_<$ under this colouring. 
\end{proof}

\begin{lem}
If $t<r$ and $F\in \ccS(r,t)$ is not complete, then $\ccS(r,t)$ does not have the
$F$-Ramsey property.
\end{lem}

\begin{proof}
Again let $x\subseteq V(F)$ be a $t$-set not contained in any edge of $G$ and let $F'$ be
obtained from $F$ by adding a new edge containing $x$ and $r-t$ new vertices. 
Define $K\in\ccS(r, t)$ to be the disjoint union of two copies of $F'$ and let $G\in\ccS(r, t)$
be obtained by applying Lemma~\ref{lem:order} to $K$. We contend that there is no
$H\in\ccS(r, t)$ with $H\lra (G)^F_2$.

To see this, let any $H\in\ccS(r, t)$ be given, order it arbitrarily to get some 
$H_<\in\ccS_<(r, t)$ and consider the red-blue colouring of $H_<$ from the proof of
Lemma~\ref{lem:34}. By our choice of~$G$ every member of $\binom{H}{G}$ contains a copy 
of $K$ that in turn contains two copies of $F$ with different colours. 
So, in particular, no $\til{G}\in\binom{H}{G}$ is monochromatic. 
\end{proof}


\begin{bibdiv}
\begin{biblist}

\bib{AH78}{article}{
   author={Abramson, Fred G.},
   author={Harrington, Leo A.},
   title={Models without indiscernibles},
   journal={J. Symbolic Logic},
   volume={43},
   date={1978},
   number={3},
   pages={572--600},
   issn={0022-4812},
   review={\MR{503795 (80a:03045)}},
   doi={10.2307/2273534},
}


\bib{BR16}{article}{
   author={Bhat, Vindya},
   author={R{\"o}dl, Vojt{\v{e}}ch},
   title={A short proof of the induced Ramsey Theorem for hypergraphs},
   journal={Discrete Math.},
   volume={339},
   date={2016},
   number={3},
   pages={1147--1149},
   issn={0012-365X},
   review={\MR{3433919}},
   doi={10.1016/j.disc.2015.09.019},
}

\bib{Bod15}{article}{
   author={Bodirsky, Manuel},
   title={Ramsey classes: examples and constructions},
   conference={
      title={Surveys in combinatorics 2015},
   },
   book={
      series={London Math. Soc. Lecture Note Ser.},
      volume={424},
      publisher={Cambridge Univ. Press, Cambridge},
   },
   date={2015},
   pages={1--48},
   review={\MR{3497266}},
}	
	
\bib{Deuber75}{article}{
   author={Deuber, W.},
   title={Generalizations of Ramsey's theorem},
   conference={
      title={Infinite and finite sets (Colloq., Keszthely, 1973; dedicated
      to P. Erd\H os on his 60th birthday), Vol. I},
   },
   book={
      publisher={North-Holland, Amsterdam},
   },
   date={1975},
   pages={323--332. Colloq. Math. Soc. J\'anos Bolyai, Vol. 10},
   review={\MR{0369127 (51 \#5363)}},
}

\bib{EHP75}{article}{
   author={Erd{\H{o}}s, P.},
   author={Hajnal, A.},
   author={P{\'o}sa, L.},
   title={Strong embeddings of graphs into colored graphs},
   conference={
      title={Infinite and finite sets (Colloq., Keszthely, 1973; dedicated
      to P. Erd\H os on his 60th birthday), Vol. I},
   },
   book={
      publisher={North-Holland, Amsterdam},
   },
   date={1975},
   pages={585--595. Colloq. Math. Soc. J\'anos Bolyai, Vol. 10},
   review={\MR{0382049 (52 \#2937)}},
}

\bib{GLR72}{article}{
   author={Graham, R. L.},
   author={Leeb, K.},
   author={Rothschild, B. L.},
   title={Ramsey's theorem for a class of categories},
   journal={Advances in Math.},
   volume={8},
   date={1972},
   pages={417--433},
   issn={0001-8708},
   review={\MR{0306010 (46 \#5137b)}},
}

\bib{HJ63}{article}{
   author={Hales, A. W.},
   author={Jewett, R. I.},
   title={Regularity and positional games},
   journal={Trans. Amer. Math. Soc.},
   volume={106},
   date={1963},
   pages={222--229},
   issn={0002-9947},
   review={\MR{0143712 (26 \#1265)}},
}

\bib{HN1}{article}{
	author={Hubi\v{c}ka, Jan},
	author={Ne{\v{s}}et{\v{r}}il, Jaroslav}, 
	title={Bowtie-free graphs have a Ramsey lift}, 
	journal={Advances in Applied Mathematics, To Appear},
	eprint={1402.2700},
}

\bib{HN2}{article}{
	author={Hubi\v{c}ka, Jan},
	author={Ne{\v{s}}et{\v{r}}il, Jaroslav}, 
	title={All those Ramsey classes (Ramsey classes with closures and forbidden homomorphisms)}, 
	eprint={1606.07979},
	note={Submitted},
}


\bib{KPT}{article}{
   author={Kechris, A. S.},
   author={Pestov, V. G.},
   author={Todorcevic, S.},
   title={Fra\"\i ss\'e limits, Ramsey theory, and topological dynamics of
   automorphism groups},
   journal={Geom. Funct. Anal.},
   volume={15},
   date={2005},
   number={1},
   pages={106--189},
   issn={1016-443X},
   review={\MR{2140630 (2007j:37013)}},
   doi={10.1007/s00039-005-0503-1},
}

\bib{Keevash}{article}{
	author={Keevash, Peter}, 
	title={The existence of designs}, 
	eprint={1401.3665},
	note={Submitted},
}

\bib{LR06}{article}{
   author={Leader, Imre},
   author={Russell, Paul A.},
   title={Sparse partition regularity},
   journal={Proc. London Math. Soc. (3)},
   volume={93},
   date={2006},
   number={3},
   pages={545--569},
   issn={0024-6115},
   review={\MR{2266959}},
   doi={10.1017/S0024611506015887},
}


\bib{NeRo75}{article}{
   author={Ne{\v{s}}et{\v{r}}il, Jaroslav},
   author={R{\"o}dl, Vojt{\v{e}}ch},
   title={Partitions of subgraphs},
   conference={
      title={Recent advances in graph theory},
      address={Proc. Second Czechoslovak Sympos., Prague},
      date={1974},
   },
   book={
      publisher={Academia, Prague},
   },
   date={1975},
   pages={413--423},
   review={\MR{0429655}},
}

\bib{NeRo1}{article}{
   author={Ne{\v{s}}et{\v{r}}il, Jaroslav},
   author={R{\"o}dl, Vojt{\v{e}}ch},
   title={Partitions of finite relational and set systems},
   journal={J. Combinatorial Theory Ser. A},
   volume={22},
   date={1977},
   number={3},
   pages={289--312},
   review={\MR{0437351 (55 \#10283)}},
}

\bib{NeRo2}{article}{
   author={Ne{\v{s}}et{\v{r}}il, Jaroslav},
   author={R{\"o}dl, Vojt{\v{e}}ch},
   title={On a probabilistic graph-theoretical method},
   journal={Proc. Amer. Math. Soc.},
   volume={72},
   date={1978},
   number={2},
   pages={417--421},
   issn={0002-9939},
   review={\MR{507350 (80a:05158)}},
   doi={10.2307/2042818},
}
		
		
\bib{NeRo3a}{article}{
   author={Ne{\v{s}}et{\v{r}}il, Jaroslav},
   author={R{\"o}dl, Vojt{\v{e}}ch},
   title={Simple proof of the existence of restricted Ramsey graphs by means
   of a partite construction},
   journal={Combinatorica},
   volume={1},
   date={1981},
   number={2},
   pages={199--202},
   issn={0209-9683},
   review={\MR{625551 (83a:05101)}},
   doi={10.1007/BF02579274},
}

\bib{NeRo5}{article}{
   author={Ne{\v{s}}et{\v{r}}il, Jaroslav},
   author={R{\"o}dl, Vojt{\v{e}}ch},
   title={Two proofs of the Ramsey property of the class of finite
   hypergraphs},
   journal={European J. Combin.},
   volume={3},
   date={1982},
   number={4},
   pages={347--352},
   issn={0195-6698},
   review={\MR{687733 (85b:05134)}},
   doi={10.1016/S0195-6698(82)80019-X},
}

\bib{NeRo6}{article}{
   author={Ne{\v{s}}et{\v{r}}il, Jaroslav},
   author={R{\"o}dl, Vojt{\v{e}}ch},
   title={Combinatorial partitions of finite posets and lattices---Ramsey
   lattices},
   journal={Algebra Universalis},
   volume={19},
   date={1984},
   number={1},
   pages={106--119},
   issn={0002-5240},
   review={\MR{748915 (86c:05025)}},
   doi={10.1007/BF01191498},
}



\bib{NeRo4}{article}{
   author={Ne{\v{s}}et{\v{r}}il, Jaroslav},
   author={R{\"o}dl, Vojt{\v{e}}ch},
   title={Strong Ramsey theorems for Steiner systems},
   journal={Trans. Amer. Math. Soc.},
   volume={303},
   date={1987},
   number={1},
   pages={183--192},
   issn={0002-9947},
   review={\MR{896015 (89b:05127)}},
   doi={10.2307/2000786},
}

\bib{Nero7}{article}{
   author={Ne{\v{s}}et{\v{r}}il, Jaroslav},
   author={R{\"o}dl, Vojt{\v{e}}ch},
   title={The partite construction and Ramsey set systems},
   note={Graph theory and combinatorics (Cambridge, 1988)},
   journal={Discrete Math.},
   volume={75},
   date={1989},
   number={1-3},
   pages={327--334},
   issn={0012-365X},
   review={\MR{1001405}},
   doi={10.1016/0012-365X(89)90097-6},
}
		
		
\bib{NeRo8}{article}{
	author={Ne{\v{s}}et{\v{r}}il, Jaroslav},
    author={R{\"o}dl, Vojt{\v{e}}ch},
	title={Ramsey partial orders from acyclic graphs}, 
	journal={Order},
	doi={10.1007/s11083-017-9433-4},
	note={To Appear},
}

\bib{The10}{article}{
   author={Nguyen Van Th{\'e}, L.},
   title={Structural Ramsey theory of metric spaces and topological dynamics
   of isometry groups},
   journal={Mem. Amer. Math. Soc.},
   volume={206},
   date={2010},
   number={968},
   pages={x+140},
   issn={0065-9266},
   isbn={978-0-8218-4711-4},
   review={\MR{2667917}},
   doi={10.1090/S0065-9266-10-00586-7},
}

\bib{The13}{article}{
   author={Nguyen Van Th{\'e}, L.},
   title={More on the Kechris-Pestov-Todorcevic correspondence: precompact
   expansions},
   journal={Fund. Math.},
   volume={222},
   date={2013},
   number={1},
   pages={19--47},
   issn={0016-2736},
   review={\MR{3080786}},
   doi={10.4064/fm222-1-2},
}


\bib{PTW85}{article}{
   author={Paoli, M.},
   author={Trotter, W. T., Jr.},
   author={Walker, J. W.},
   title={Graphs and orders in Ramsey theory and in dimension theory},
   conference={
      title={Graphs and order},
      address={Banff, Alta.},
      date={1984},
   },
   book={
      series={NATO Adv. Sci. Inst. Ser. C Math. Phys. Sci.},
      volume={147},
      publisher={Reidel, Dordrecht},
   },
   date={1985},
   pages={351--394},
   review={\MR{818500}},
}

\bib{Ramsey30}{article}{
   author={Ramsey, Frank Plumpton},
   title={On a Problem of Formal Logic},
   journal={Proceedings London Mathematical Society},
   volume={30},
   date={1930},
   number={1},
   pages={264--286},
         doi={10.1112/plms/s2-30.1.264},
}

\bib{Rodl73}{unpublished}{
	author={R{\"o}dl, Vojt{\v{e}}ch}, 
	title={The dimension of a graph and generalized Ramsey numbers}, 
	note={Master's Thesis, Charles University, Praha, Czechoslovakia},
	date={1973},
}

\bib{Rodl76}{article}{
    author = {{R\"odl}, Vojt\v{e}ch},
    title = {A generalization of the Ramsey theorem},
    conference={
    		title={Graphs, Hypergraphs, Block Syst.},
			address={Proc. Symp. comb. Anal., Zielona Gora},
			date={1976},
	},
    date={1976},
    pages={211--219},
    review={Zbl. 0337.05133},
}

\bib{Sh329}{article}{
   author={Shelah, Saharon},
   title={Primitive recursive bounds for van der Waerden numbers},
   journal={J. Amer. Math. Soc.},
   volume={1},
   date={1988},
   number={3},
   pages={683--697},
   issn={0894-0347},
   review={\MR{929498}},
   doi={10.2307/1990952},
}
	
\bib{Solecki}{article}{
   author={Solecki, S\l awomir},
   title={Recent developments in finite Ramsey theory: foundational 
			aspects and connections with dynamics},
   conference={
      title={Proceedings of the International Congress of Mathematicians},
      address={Seoul},
      date={2014},
   },
   book={
      publisher={Kyung Moon Sa Co. Ltd.},
   },
   date={2014},
   volume={2},
   pages={103--115},
}

			
\end{biblist}
\end{bibdiv}
\end{document}